\documentclass[aihp]{imsart}

\RequirePackage{amsthm,amsmath,amsfonts,amssymb}
\RequirePackage[numbers,sort&compress]{natbib}
\RequirePackage[colorlinks,citecolor=blue,urlcolor=blue]{hyperref}
\RequirePackage{graphicx}

\startlocaldefs
\theoremstyle{plain}

\newtheorem{theorem}{Theorem}[section]
\newtheorem{lemma}[theorem]{Lemma}
\newtheorem{definition}[theorem]{Definition}
\newtheorem{example}[theorem]{Example}

\theoremstyle{remark}
\newtheorem{remark}{Remark}

\endlocaldefs

\begin{document}
\begin{frontmatter}
\title{Onsager-Machlup functional for stochastic differential equations
	with time-varying noise}
\runtitle{OM functional for SDEs with time-varying noise}

\begin{aug}
	\author[A]{\inits{F.}\fnms{Xinze}~\snm{Zhang}\ead[label=e1]{zhangxz24@mails.jlu.edu.cn}},
	\author[A,B]{\inits{S.}\fnms{Yong}~\snm{Li}\ead[label=e2]{liyong@jlu.edu.cn}}\thanks{corresponding author.}.
	\address[A]{School of Mathematics, Jilin University, ChangChun, People's Republic of China}
	
	\address[B]{School of Mathematics and Statistics, Center for Mathematics and Interdisciplinary Sciences,\\    ~~~ Northeast Normal University, ChangChun, People's Republic of China\\
		\printead{e1,e2}}
	
\end{aug}

\begin{abstract}
This paper is devoted to studying the Onsager-Machlup functional for stochastic differential equations with time-varying noise of the $\alpha$-Hölder, ${0 < \alpha < \frac{1}{4}}$,
\begin{equation}
	\begin{aligned}
		\mathrm{d} X_t = f(t, X_t)\mathrm{d} t + g(t)\mathrm{d}W_t.
	\end{aligned}\notag
\end{equation}
Our study focuses on scenarios where the diffusion coefficient \(g(t)\) exhibits temporal variability, starkly contrasting the conventional assumption of a constant diffusion coefficient in the existing literature. This variance brings some complexity to the analysis. Through this investigation, we derive the Onsager-Machlup functional, which acts as the Lagrangian for mapping the most probable transition path between metastable states in stochastic processes affected by time-varying noise.  This is done by introducing new measurable norms and applying an appropriate version of the Girsanov transformation. To illustrate our theoretical advancements, we provide numerical simulations, including cases of a one-dimensional SDE and a fast-slow SDE system, which demonstrate the application to multiscale stochastic volatility models, thereby highlighting the significant impact of time-varying diffusion coefficients.
\end{abstract}



\end{frontmatter}
\section*{Statements and Declarations}
\begin{itemize}
	\item Ethical approval\\
	Not applicable.
	
	\item Competing interests\\
	The authors have no competing interests as defined by IOP PUBLISHING LTD, or other interests that might be perceived to influence the results and/or discussion reported in this paper.
	
	\item Authors' contributions\\
	X. Zhang wrote the main manuscript text and prepared figures 1-3. Y. Li provided ideas and explored specific research methods. All authors reviewed the manuscript.
	
	\item Funding\\
	The second author was supported by National Natural Science Foundation of China (Grant No. 12071175)
	
	\item Availability of data and materials\\
	Not applicable.
\end{itemize}
\section{Introduction}
In the field of stochastic dynamical systems, it is crucial to understand how systems transition between metastable states. This involves studying Stochastic Differential Equations (SDEs) to identify deterministic and significant quantities such as solution paths, mean exit time, and escape probability. An essential tool in this context is the Onsager-Machlup function, which serves as a deterministic quantity used to characterize the most probable transition paths between metastable states in a system. 

The concept of the Onsager-Machlup function was initially proposed by Onsager and Machlup \cite{1,2} in 1953, where they defined it as the probability density functional for a diffusion process characterized by linear drift and constant diffusion coefficients. Following this foundational work, Tizsa and Manning \cite{3} extended the application scope of the Onsager-Machlup function to nonlinear equations in 1957.  In parallel, Stratonovich \cite{4} introduced a rigorous mathematical approach in the same year, significantly enriching the theoretical underpinnings of the Onsager-Machlup function.

The Onsager-Machlup functional for SDEs driven by Brownian motion has been the subject of extensive research, as documented in numerous studies \cite{5,6,7,8,9,10}. In recent years, Moret and Nualart \cite{12} studied the Onsager-Machlup functional of SDEs driven by fractional Brownian motion. Here, hurst index $H$ is divided into two situations: $\frac{1}{4} < H < \frac{1}{2}$ and $\frac{1}{2} < H < 1$; Chao and Duan \cite{14} derived the Onsager-Machlup function for a class of stochastic dynamical systems under (non-Gaussian) Lévy noise as well as (Gaussian) Brownian noise, and examined the corresponding most probable paths; Li and Li \cite{11} demonstrated that the $\Gamma$-limit of the Onstage-Machlup functional on the space of curves is the geometric form of the Freidlin-Wentzell functional in a proper time scale; Du et al. \cite{22} proved that as time approaches infinity, an unbounded OM functional minimum sequence contains a convergent subsequence in the curve space, and the graph limit of this minimum subsequence is the extremum of the action functional. And based on this, a geometric minimization algorithm for energy climb is proposed; Liu et al. \cite{15} obtained the Onsager-Machlup functional for McKean-Vlasov SDEs in a class of norms that dominate $L^2([0,1],\mathbb{R}^d)$. Carfagnini and Wang \cite{25} proved that the Loewner energy can be interpreted as an Onsager-Machlup functional for the SLE\(_\kappa\) loop measure for any fixed \( \kappa \in (0, 4] \).

Nevertheless, there has been few research on SDEs where the diffusion coefficients are not constants but vary over time. Bardina et al. \cite{13} explored additive noise in SDEs, involving the diffusion term coefficient $B$, where $B$ is a non-negative bounded linear operator but not depending on time $t$. Coulibaly-Pasquier \cite{23} extended the results of Capitaine \cite{24} and computed the Onsager-Machlup functional for a non-homogeneous elliptic diffusion process on a Riemannian manifold. Specifically, on a Riemannian manifold equipped with a metric \( g = (\sigma \sigma^*) \), Coulibaly-Pasquier proved that let \( X_t(x_0) \) be an \( L_t \) diffusion process starting at point \( x_0 \), where \( L_t = \frac{1}{2} \Delta_t + Z(t, \cdot) \). The corresponding Onsager-Machlup functional is given by:
\[
\begin{aligned}
	OM(t,\varphi, \dot{\varphi}) = \int_{0}^{1} \| Z(t, \varphi) - \dot{\varphi} \|_{g(t)}^2 +  \operatorname{div}_{g(t)}(Z(t, \varphi)) - \frac{1}{6} R_{g(t)}(\varphi) + \frac{1}{2} \operatorname{trace}_{g(t)}(\dot{g}(t)) \,{\rm d}t,
\end{aligned}
\]
where \( \operatorname{div}_{g(t)} \), \( R_{g(t)} \), and \( \operatorname{trace}_{g(t)} \) denote the divergence operator, the scalar curvature, and the trace of a matrix with respect to the metric \( g(t) \), respectively.

It is important to note that \( \sigma \) represents the diffusion induced by the Riemannian structure. As Capitaine states in Section 4 of \cite{24}, "When \( \sigma \) is the identity matrix, the Riemannian structure induced by the diffusion is the Euclidean one." In other words, as shown in Theorem 2 of \cite{24}, in the case of Euclidean space, \( \sigma \) is the identity matrix, which corresponds to the following stochastic differential equation:
\[
\begin{aligned}
	\mathrm{d} X_t = b(t, X_t)\, \mathrm{d} t + \mathrm{d} W_t.
\end{aligned}
\]
Therefore, the results are not directly applicable to time-varying stochastic differential equations in Euclidean space. This observation significantly motivates our exploration of the Onsager-Machlup functionals for SDEs with time-varying noise in \( \mathbb{R}^n \), described by:
\begin{equation}\label{1.1}
	\begin{aligned}
		\mathrm{d} X_t = f(t, X_t)\, \mathrm{d} t + g(t)\, \mathrm{d} W_t,
	\end{aligned}
\end{equation}
where \( f \in C^2_b\left([0, 1] \times \mathbb{R}^n, \mathbb{R}^n \right) \) and \( g \in C^2_b\left([0, 1], \mathbb{R}^n \right) \), with \( C^2_b\left(A, B\right) \) denoting the space of second-derivative continuous bounded mappings from \( A \) to \( B \), and \( W_t \) being an \( n \)-dimensional Brownian motion. Under the \( \alpha \)-Hölder norm, the Onsager-Machlup functional for Equation \(\eqref{1.1}\) is given by:
\[
OM(\varphi, \dot{\varphi}) = \int_{0}^{1} \left| g(t)^{-1} \cdot \left(\dot{\varphi}_t - f(t, \varphi_t)\right) \right|^2 \,{\rm d}t + \int_{0}^{1} \operatorname{div}^{g}_{x} f(t, \varphi_t) \,{\rm d}t,
\]
where our metric is independent of the diffusion, which represents a fundamental difference from the paper \cite{23}, where the detailed derivation can be found in Theorem 4.3.

This paper is devoted to deriving the Onsager-Machlup functional for SDEs driven by time-varying noise, covering both $1$-dimensional and $n$-dimensional cases. The 1-dimensional scenario, due to its favorable properties, enables us to achieve superior results under a wider range of conditions. Conversely, the analysis in the higher-dimensional context demands careful matrix analysis and computations, necessitating stricter conditions on the coefficients to derive meaningful outcomes. Ultimately, we successfully derive the Onsager-Machlup function, maintaining its interpretation as a Lagrangian function. This achievement enables us to delineate the most probable transition paths surrounding smooth trajectories within diffusion processes.

The challenge of this study lies in the diffusion coefficient $g$ in SDEs being tied to the time variable $t$. In the $n$-dimensional case, we draw inspiration from \cite{12}, define a new norm, and prove its measurability. To apply the Girsanov transformation, it is essential to ensure that $g(t)$ is an invertible $n \times n$ matrix for all $0 \leq  t \leq 1$. Subsequent analysis and computations leverage techniques such as Taylor expansion, matrix trace, and martingales. Through extensive estimation, we ultimately achieve the desired functional.

The structure of the paper is as follows. In Section 2, we collect some known notions and facts. Namely we give the definitions of measurable norm and Onsager-Machlup functional, and introduce several important technical lemmas and theorems. In Section 3, we study the Onsager-Machlup functional for $1$-dimensional SDEs. In Section 4, we further study the Onsager-Machlup functional for $n$-dimensional SDEs. In Section 5, we test some specific examples to illustrate our results.

\section{Preliminaries}

\subsection{Approximate limits in Wiener space}
In this section, we recall some fundamental definitions and results concerning approximate limits in Wiener space. Specifically, we focus on the measurable semi-norm, which pertains to the exponentials of random variables in the first and second Wiener chaos (reference \cite{16}).

Let $W = \left\{ W_t, ~t \in [0, 1] \right\}$ be a Brownian motion (Wiener process) defined in the complete filtered probability space $(\Omega, \mathcal{F}, \left\{ \mathcal{F}_t \right\}_{t \geq 0}, \mathbb{P})$. Here, $\Omega$ represents the space of continuous functions vanishing at zero, and $\mathbb{P}$ denotes the Wiener measure. Let $H := L^2([0,1], \mathbb{R}^n)$ be a Hilbert space and $H^1$ be the Cameron-Martin space defined as follows:
\begin{displaymath}
	H^1 := \left\{ f : [0, 1] \to \mathbb{R}^n \in H^1 ~\big|~f(0) = 0, f ~ \text{is absolutely continuous functions and} ~ f^{\prime} \in H \right\}.
\end{displaymath}
The scalar product in $H^1$ is defined as follows:
\begin{displaymath}
	\langle f, g \rangle_{H^1} = \langle f^{\prime}, g^{\prime} \rangle_{H}
\end{displaymath}
for all $f, g \in H^1$.
Let $\mathcal{P}:H^1 \to H^1$ be an orthogonal projection with $dim \mathcal{P}H^1 < \infty$ and the specific expression
\begin{displaymath}
	\mathcal{P}h = \sum_{i = 1}^{n} \langle h_i, f \rangle h_i,
\end{displaymath}
where $(h_1, ..., h_n)$ is a set of orthonormal basis in $\mathcal{P}H^1$. In addition, we can also define the
$H^1$-valued random variable
\begin{displaymath}
	\mathcal{P}^W = \sum_{i = 1}^{n} \bigg( \int_{0}^{1} {h_i^{\prime}} \,{\rm d}W_s \bigg) h_i,
\end{displaymath}
where $\mathcal{P}^W$ does not depend on $(h_1, ..., h_n)$.
\begin{definition}\label{definition 2.1}
	We say that a sequence of orthogonal projections $\mathcal{P}_n$ on $H^1$ is an approximating sequence of projections, if $dim \mathcal{P}_n H^1 < \infty$ and $\mathcal{P}_n$ converges strongly to the identity operator $I$ in $H^1$ as $n \to \infty$.
\end{definition}

\begin{definition}\label{definition 2.2}
	We say that a semi-norm $\mathcal{N}$ on $H^1$ is measurable, if there exists a random variable $\tilde{\mathcal{N}}$, satisfying $\tilde{\mathcal{N}} < \infty $ a.s, such that for any approximating sequence of projections $\mathcal{P}_n$ on $H^1$, the sequence $\mathcal{N}(\mathcal{P}^W_n)$ converges to $\tilde{\mathcal{N}}$ in probability and $\mathbb{P}(\tilde{\mathcal{N}} \leq \epsilon) > 0$ for any $\epsilon > 0$. Moreover, if $\mathcal{N}$ is a norm on $H^1$, then we call it a measurable norm.
\end{definition}

For proving the measurability of the semi-norm defined in this paper, it is necessary to introduce the following lemma (see \cite{17}).
\begin{lemma}\label{lemma 2.2}
	Let $\mathcal{N}_n$ be a nondecreasing sequence of measurable semi-norms. Suppose that $\tilde{\mathcal{N}} := \mathbb{P}\text{-}\!\lim\limits_{n \to \infty} \tilde{\mathcal{N}}_n$ exists and $\mathbb{P}(\tilde{\mathcal{N}} \leq \epsilon) > 0$ for any $\epsilon > 0$. In addition, if the limit $\lim\limits_{n \to \infty} \mathcal{N}_n$ exists on $H^1$, then $\mathcal{N} := \lim\limits_{n \to \infty} \mathcal{N}_n$ is a measurable semi-norm.
\end{lemma}
\begin{definition}\label{definition 2.4}
	Let $f$ be a function defined on $\Omega$. For $0 < \alpha < 1$, we introduce Hölder norm ($\alpha$-Hölder)
	\begin{displaymath}
		\Vert f \Vert_{\alpha; \Omega} = \Vert f \Vert_{0; \Omega} + \left[ f \right]_{\alpha; \Omega},
	\end{displaymath}
	where $\Vert f \Vert_{0; \Omega}$ represents the supremum norm of $f$ on $\Omega$, and $\left [f \right]_{\alpha; \Omega}$ represents the Hölder semi-norm of $f$ on $\Omega$. The specific expression is as follows:
	\begin{displaymath}
		\Vert f \Vert_{0; \Omega}  = \sup\limits_{x \in \Omega}\vert f(x) \vert, \quad
		\left[ f \right]_{\alpha; \Omega} = \sup\limits_{x, y \in \Omega, x \neq y} \frac{\left| f(x) - f(y) \right|}{\vert x - y \vert^{\alpha}}.
	\end{displaymath}
\end{definition}
Throughout this paper, unless otherwise stated, the norm \( \Vert \cdot \Vert \) denotes the Hölder norm \( \Vert \cdot \Vert_{\alpha} \). We define the space
\[
H^1_{x_0} := \{ f(t) \mid f(t) - x(0) \in H^1_0 \}.
\]
Using the aforementioned norm, we induce the uniform topology on \( H^1_{x_0} \), resulting in the Borel \( \sigma \)-field \( \mathcal{B}_{x_0} \) on \( H^1_{x_0} \).

A subset \( I_n \) of \( H^1_{x_0} \) is said to have the form
\[
I_n = \{ x \in H^1_{x_0} \mid (x(t_1), \ldots, x(t_n)) \in E \},
\]
where \( s < t_1 < \cdots < t_n \leq u \), and \( E \) is a Borel set in \( \mathbb{R}^n \). Such a set is called an \( n \)-dimensional cylinder set. The collection of all \( n \)-dimensional cylinder sets forms a \( \sigma \)-field, and the class of all finite-dimensional cylinder sets forms a field, denoted by \( I \). It is shown in \cite{28} that the \( \sigma \)-field \( \sigma(I) \) generated by \( I \) is precisely the Borel \( \sigma \)-field \( \mathcal{B}_{x_0} \), i.e.,
\[
\sigma(I) = \mathcal{B}_{x_0}.
\]

Next, we define the measure \( \mu_x \) on \( \mathcal{B}_{x_0} \), which is induced by the diffusion process $\eqref{1.1}$, as
\[
\mu_x(B) = P\left( \{ \omega \in \Omega \mid X_t(\omega) \in B \} \right), \quad B \in \mathcal{B}_{x_0}.
\]
If \( B = I_n \), then
\[
\mu_x(I_n) = P_{x_0}^n\left( \{ (x(t_1), \ldots, x(t_n)) \in E \} \right),
\]
where \( E \) is the Borel set in \( \mathbb{R}^n \) under consideration, and \( P_{x_0}^n \) denotes the \( n \)-dimensional probability measure of the diffusion process \( X_t \).

\subsection{Onsager-Machlup functional}
In the problem of finding the most probable path of a diffusion process, the probability of a single path is zero. Instead, we can search for the probability that the path lies within a certain region, which could be a tube along a differentiable function. This tube is defined as
\[
K(\varphi, \epsilon) = \{ x \in H^1_{x_0} \mid \varphi \in H^1_{x_0}, \|x - \varphi\| \leq \epsilon, \epsilon > 0 \}.
\]
Once \( \epsilon > 0 \) is given, the probability of the tube can be expressed as
\[
\mu_x(K(\varphi, \epsilon)) = P\left( \{ \omega \in \Omega \mid X_t(\omega) \in K(\varphi, \epsilon) \} \right),
\]
allowing us to compare the probabilities of the tubes for all \( \varphi\in H^1_{x_0} \), since \( K(\varphi, \epsilon) \in \mathcal{B}_{x_0} \).

Thus, the Onsager-Machlup function can be defined as the Lagrangian function that gives the most probable tube. We now introduce the definitions of the Onsager-Machlup function and the Onsager-Machlup functional.
\begin{definition}
	Consider a tube surrounding a reference path $\varphi_t$ with initial value $\varphi_0 = x$ and $\varphi_t - x$ belongs to $H^1$. Assuming $\epsilon$ is given and small enough, we estimate the probability that the solution process $X_t$ is located in that tube as:
	\begin{displaymath}
		\mathbb{P} \left\{ \Vert X - \varphi\Vert \leq \epsilon\right\}  \propto C(\epsilon) {\rm exp} \left\{ -\frac{1}{2} \int_{0}^{1} {OM(t, \varphi, \dot{\varphi})} \,{\rm d}t \right\},
	\end{displaymath}
	where $\propto$ denotes the equivalence relation for $\epsilon$ small enough. Then we call the integrand $OM(t, \varphi, \dot{\varphi})$ the Onsager-Machulup function and also call integral $\int_{0}^{1} {OM(t, \varphi, \dot{\varphi})} \,{\rm d}t$ the Onsager-Machulup functional. In analogy to classical mechanics, we also refer to the Onsager-Machulup function as the Lagrangian function and the Onsager-Machulup functional as the action functional.
\end{definition}

\subsection{Quasi-translation invariant measure}
It is well known that the uniqueness of the Lebesgue measure \(\mu_L\) on \(\mathbb{R}^n\) is characterized by its translation invariance. Specifically, if \(\mathcal{T}\) is a translation on \(\mathbb{R}^n\), then for each \(E \in \mathcal{B}^n\), the following holds:
\begin{equation*}
	\mathcal{T} E \in \mathcal{B}^n \quad \text{and} \quad \mu_L(\mathcal{T}E) = \mu_L(E)
	\label{2.19}
\end{equation*}
However, in the space \(H^1_{x_0}\), such a one-to-one mapping corresponds to a translation by any function \(\varphi(t) \in H^1_{x_0}\). It is evident that a translation-invariant measure \(\mu_x\) on \(H^1_{x_0}\) does not exist \cite{26}. To address this, we introduce a weaker concept from \cite{5}, the quasi translation invariant measure, defined as follows:

\begin{definition}\label{D2.5}
	Let \(\mathcal{T}\) be a transformation \(T: H^1_{x_0} \rightarrow H^1_{x_0}\) such that
	\begin{equation*}
		\mathcal{T} x = x + \varphi,
		\label{2.20}
	\end{equation*}
	where \( \varphi \in H^1_{x_0} \) is absolutely continuous, and its derivative belongs to the space \( H \). Consider the diffusion process \(X_t\) and its translated process
	\begin{equation*}
		\mathcal{T} X_t = X_t + \varphi(t).
		\label{2.21}
	\end{equation*}
	If the induced measures \(\mu_X\) and \(\mu_{\mathcal{T}X}\) are equivalent, then \(\mu_X\) and \(\mu_{\mathcal{T}X}\) are called quasi translation invariant measures.
\end{definition}

For stochastic differential equations with constant coefficients, \cite{5,14} present several results concerning quasi translation invariant measures. In this work, we extend these results to stochastic differential equations with time-dependent coefficients.
\begin{lemma}\label{L2.7}
	The stochastic differential equation $\eqref{1.1}$ induces a quasi translation invariant measure $\mu_X$, and the Radon–Nikodym derivative is given by
	\begin{equation*}
		\frac{{\rm d} \mu_{\mathcal{T}X}}{{\rm d} \mu_X} [X_t(\omega)] = \exp \left\{ \int_s^u a(t, X_{t}, \varphi(t)) \,{\rm d}B_t - \frac{1}{2} \int_s^u \left( a\left( X_{t}, \varphi(t)\right) \right)^2 \,{\rm d}t \right\},
		\label{3.13}
	\end{equation*}
	where
	\begin{equation*}
		a(t, X_{t}, \varphi(t)) = \frac{f(t, X - \varphi) - f(t,X) + \dot{\varphi}}{g(t)}.
	\end{equation*}
\end{lemma}
\begin{proof}
	Applying the translation \(\mathcal{T}\) to equation \eqref{1.1} in accordance with Definition \ref{D2.5}, we obtain:
	\begin{equation*}
		\begin{aligned}
			{\rm d} \mathcal{T}X_t &={\rm d} \left( X_t + \varphi(t)\right) \\
			&= \left( f(t, X_t) + \dot{\varphi}(t)\right) \,{\rm d} t + g(t) \,{\rm d}W_t\\
			&= \left( f(t, \mathcal{T} X_t - \varphi(t)) + \dot{\varphi}(t)\right) \,{\rm d} t + g(t) \,{\rm d} W_t
		\end{aligned}
	\end{equation*}
	Since the diffusion process remains unchanged under the translation \(\mathcal{T}\), the proof is completed by Girsanov's theorem, as referenced in \cite{27}.
\end{proof}

We denote \( J_X[X_t, \varphi(t)] \) as the Radon-Nikodym derivative of \( \mu_{\mathcal{T}X} \) with respect to \( \mu_X \). If we replace \( \varphi(t) \) by \( -\varphi(t) \) in these functionals, they refer to the Radon-Nikodym derivative of \( \mu_{\mathcal{T}^{-1}X} \) with respect to \( \mu_X \). The following lemma will demonstrate the significance of the aforementioned quasi-translation invariant measure.
\begin{lemma}\label{L2.8}
	If we consider the translation map \( \mathcal{T} \) defined in Definition \ref{D2.5} and define \( \mathcal{M} = \{ E \in \mathcal{B}_{x_0} : \mathcal{T}^{-1}E \in \mathcal{B}_{x_0} \} \), then \( T \) is measurable on \( \mathcal{M} \). If \( F(x) \) is a measurable functional on \( H^1_{x_0} \) and \( \mu_x \) is a quasi-translation invariant measure, the following equation holds on \( \mathcal{M} \):
	\[
	\int_A F(y) d\mu_x(y) = \int_{\mathcal{T}^{-1}A} F(x + \varphi) J_X[x, -\varphi] d\mu_x(x).
	\]
\end{lemma}
\begin{proof}
	Based on the definition of \( J_X[x, \varphi] \) as the Radon-Nikodym derivative of \( \mu_{T^{-1}X} \) with respect to \( \mu_X \), we have the following relationship:
	\[
	J_X[x, -\varphi] \, d\mu_X(x) = d\mu_{T^{-1}X}(x).
	\]
	Now, for any set \( E \in \mathcal{B}_{x_0} \), the following equality holds:
	\[
	\mu_{T^{-1}X}(T^{-1}E) = P\left( \{ \omega \mid T^{-1} X_t(\omega) \in T^{-1} E \} \right) = P\left( \{ \omega \mid X_t(\omega) \in E \} \right) = \mu_X(E),
	\]
	which leads to the desired result in lemma $\ref{L2.8}$. 
\end{proof}
\begin{remark}
	Here, \( g(t) \) must be a function solely dependent on time and not on the spatial variable \( x \), otherwise, Lemma 2.7 would not hold. In other words, if the diffusion coefficient \( g(t, x) \) depends on both time and space, the induced measure of equation $\eqref{1.1}$ would not be quasi translation invariant.
\end{remark}

\subsection{Technical lemmas and theorems}
In this section, we will introduce several commonly utilized technical lemmas and theorems. Throughout this paper, if not mentioned otherwise, $\mathbb{E} \left(A \big| B\right)$ represents the conditional expectation of $A$ under $B$. 

When we derive the Onsage-Machup functional of SDEs with additive noise, the following lemma is the most basic one, as it ensures that we handle each term separately. Its proof can be found in \cite{18}.
\begin{lemma}\label{lemma 2.3}
	For a fixed integer $N \geq 1$, let $X_1, ..., X_N \in \mathbb{R}$ be $N$ random variables defined on $(\Omega, \mathcal{F}, \left\{ \mathcal{F}_t \right\}_{t \geq 0}, \mathbb{P})$ and $\left\{D_{\epsilon}; \epsilon > 0 \right\}$ be a family of sets in $\mathcal{F}$. Suppose that for any $c \in \mathbb{R}$ and any $i = 1, ..., N$, we have
	\begin{displaymath}
		\limsup\limits_{\epsilon \to 0} \mathbb{E}\left({\rm exp}\left\{ c X_i \right\}\big|D_{\epsilon} \right) \leq 1.
	\end{displaymath}
	Then
	\begin{displaymath}
		\limsup\limits_{\epsilon \to 0} \mathbb{E}\left({\rm exp}\left\{ \sum_{i = 1}^{N}c X_i \right\} \big|D_{\epsilon} \right)= 1.
	\end{displaymath}
\end{lemma}
The following two theorems are fundamental parts of calculating Onsage-Machup functional. Their proofs can be found in \cite{16}.
\begin{theorem}\label{theorem 2.4}
	Let $\mathcal{N}$ be a measurable norm on $H^1$. For any $f \in L^2([0, 1])$, we have
	\begin{displaymath}
		\lim\limits_{\epsilon \to 0}\mathbb{E} \left( {\rm exp} \left\{ {\int_{0}^{1} {f(s)} \,{\rm d}W_s } \right\} \big| \tilde{\mathcal{N}} \leq \epsilon \right) = 1,
	\end{displaymath}
where $\tilde{\mathcal{N}}$ is defined by Definition $\ref{definition 2.2}$.
\end{theorem}
\begin{definition}
	 We say that an operator $S : H \to H$ is nuclear, if
	\begin{displaymath}
		\sum_{n = 1}^{\infty} \left|\langle  Se_n, k_n  \rangle \right| < \infty,
	\end{displaymath}
	for any orthonormal sequences $B_1 = (e_n)_n$ and $B_2 = (k_n)_n$ in $H$.
\end{definition}
We define the trace of a nuclear operator $S$ as follows:
\begin{displaymath}
	Tr.S = \sum_{n = 1}^{\infty} \langle  Se_n, e_n  \rangle
\end{displaymath}
for any orthonormal sequence $B = (e_n)_n$ in $H$. The definition of trace is independent of the orthonormal sequences we choose. For a given symmetric function $f \in L^2([0, 1]^2)$, the Hilbert-Schmidt operator $S(f) : H \to H$ is defined as:
\begin{displaymath}
	\left(S \left(f \right) \right)(h)(t) = \int_{0}^{t} {f(t, s)h(s)} \,{\rm d}s
\end{displaymath}
is nuclear if $\sum_{n = 1}^{\infty} \langle  Se_n, e_n  \rangle < \infty$ for any orthonormal sequence $B = (e_n)_n$ in $H$. When the function $f$ is continuous and the operator $S(f)$ is nuclear, the trace of $f$ has the following expression(see \cite{19}):
\begin{displaymath}
	Tr.f := Tr.S(f) = \int_{0}^{1} {f(t, t)} \,{\rm d}t.
\end{displaymath}
Furthermore, when ${f(s, t)}$ is a continuous $n \times n$ covariance kernel in the square $0 \leq s, t \leq 1$, the corresponding operator $S$ is nuclear and the expression for its trace is as follows:
\begin{displaymath}
	Tr.f = Tr.S(f) = \int_{0}^{1} {Tr.f(t, t)} \,{\rm d}t.
\end{displaymath}
\begin{theorem}\label{theorem 2.5}
	Let $f$ be a symmetric function in $L^2([0, 1]^2)$ and let $\mathcal{N}$ be a measurable norm. If $S(f)$ is nuclear, then
	\begin{displaymath}
		\lim\limits_{\epsilon \to 0} \mathbb{E} \left(  {\rm exp}\left\{ \int_{0}^{1} {\int_{0}^{1} {f(s, t)} \,{\rm d}W_s} \,{\rm d}W_t   \right\} \big| \tilde{\mathcal{N}} \leq \epsilon \right)  = e^{-Tr.(f)}.
	\end{displaymath}
\end{theorem}

The following theorem and lemma are about the probability estimation of Brownian motion balls, which are the basis for the theorem in this article. The proof of the theorem can be found in \cite{20}.
\begin{theorem}\label{theorem 2.6}
	Let $\left\{W(t): t \geq 0\right\}$ be a sample continuous Brownian motion in $\mathbb{R}$ and set 
	\begin{displaymath}
		\Phi_{\alpha}(\epsilon) = {\rm log}\mathbb{P}\left(\Vert W \Vert_{\alpha} \leq \epsilon \right).
	\end{displaymath}
	If $0 < \alpha < \frac{1}{2}$, then 
	\begin{displaymath}
		\lim\limits_{\epsilon \to 0} \epsilon^{\frac{2}{1 - 2\alpha}} \Phi_{\alpha}(\epsilon) = -C_{\alpha}
	\end{displaymath}
	exists with
	\begin{displaymath}
		2^{-\frac{2\left(1 - \alpha\right)}{\left(1 - 2\alpha\right)}} \Lambda_{\alpha} \leq C_{\alpha} \leq \left(2^{-\frac{1}{2}} \left( 2^{\alpha} - 1\right) \left( 2^{1 - \alpha} - 1\right)\right)^{-\frac{2(1 - \alpha)}{(1 - 2\alpha)}} \Lambda_{\alpha},
	\end{displaymath}
	where
	\begin{displaymath}
		\Lambda_{\alpha} = \left(\frac{2}{\pi}\right)^{\frac{1}{2}} \int_{0}^{\infty} {\frac{x^{\frac{2}{1 - 2\alpha}} e^{-\frac{x^2}{2}}}{1 - G(x)}} \,{\rm d}x \qquad\text{and}\qquad
		G(x) = \left(\frac{2}{\pi}\right)^{\frac{1}{2}} \int_{x}^{\infty} {e^{-\frac{y^2}{2}}} \,{\rm d}y.
	\end{displaymath}
\end{theorem}
\begin{lemma}\label{lemma 2.7}
	Let $g(t) \in  C^2_b([0, 1], \mathbb{R}^d)$, $ M > m > 0$ and assume $m \leq g(t) \leq M$, for any $0 \leq t \leq 1$. We have
	\begin{displaymath}
		mW_t\leq \int_{0}^{t} {g(s)} \,{\rm d}W_s \leq MW_t,
	\end{displaymath}
	for any $0 \leq t \leq 1$. According to Theorem $\ref{theorem 2.6}$, we have
	\begin{displaymath}
		\lim\limits_{\epsilon \to 0} \mathbb{P}\left(\Vert \int_{0}^{1} {g(t)} \,{\rm d}W_t \Vert_{\alpha} \leq  \epsilon \right) \geq \lim\limits_{\epsilon \to 0} \mathbb{P}\left(\Vert W \Vert_{\alpha} \leq \frac{\epsilon}{M} \right)
		\geq e^{-c \left( \frac{\epsilon}{M} \right)^{-\frac{2}{1 - 2\alpha}}},
	\end{displaymath}
	where $c = \left(2^{-\frac{1}{2}} \left( 2^{\alpha} - 1\right) \left( 2^{1 - \alpha} - 1\right)\right)^{-\frac{2(1 - \alpha)}{(1 - 2\alpha)}} \Lambda_{\alpha}$.
\end{lemma}

We define the following norms on $H^1$:
\begin{displaymath}
	\mathcal{N}_{g, 0}(h) := \sup\limits_{t \in [0, 1]}\left| \int_{0}^{t} {g(s) h^{\prime}(s)} \,{\rm d}s \right|,
\end{displaymath}
\begin{displaymath}
	\mathcal{N}_{g, \alpha}(h) := \sup\limits_{t \in [0, 1]} \frac{\vert \int_{0}^{t} {g(s) h^{\prime}(s)} \,{\rm d}s - \int_{0}^{r} {g(s) h^{\prime}(s)} \,{\rm d}s \vert}{\vert t - r \vert^{\alpha}}, \quad 0 < \alpha < \frac{1}{4},
\end{displaymath}
\begin{displaymath}
	\mathcal{N}_{g}(h) := \mathcal{N}_{g, 0}(h) + \mathcal{N}_{g, \alpha}(h), \quad 0 < \alpha < \frac{1}{4}.
\end{displaymath}
In order to apply the above theorems in this paper, we need the following lemma.
\begin{lemma}\label{lemma 2.8}
	$\mathcal{N}_g$ with $0 < \alpha <\frac{1}{2}$ is measurable norms and we have $\tilde{\mathcal{N}}_g = \Vert  \int_{0}^{1} {g(t)} \,{\rm d}W_t \Vert_{\alpha}$.
\end{lemma}
\begin{proof}
	According to the properties of norm and semi-norm, it suffices to show that $\mathcal{N}_{g, 0}$ is a measurable norm and $\mathcal{N}_{g, \alpha}(h)$ is a measurable semi-norm. Below, we will only prove that $\mathcal{N}_{g, 0}$ is a measurable norm, because the proof of $\mathcal{N}_{g, \alpha}$ is similar to $\mathcal{N}_g$. Fix $t \in [0, 1]$ and define the continuous linear functional $\varphi_t : H^1 \to \mathbb{R}$ as follows:
	\begin{displaymath}
		\varphi_t(h) = \int_{0}^{1} {g(s)h^{\prime}(s)} \,{\rm d}s.
	\end{displaymath}
	Then we can show that $\vert \varphi_t(\cdot) \vert$ represents a measurable norm. Define the measurable norms $\mathcal{N}_n(h) = \sup\limits_{0 \leq j \leq 2^n} \vert \varphi_{j2^{-n}}(h) \vert$. In addition, we have the following convergence regarding limits $n \to \infty$: 
	\begin{displaymath}
		\tilde{\mathcal{N}}_n = \sup\limits_{0 \leq j \leq 2^n} \bigg\vert \int_{0}^{j2^{-n}} {g(t)} \,{\rm d}W_t \bigg\vert \stackrel{\mathbb{P}}{\longrightarrow} \sup\limits_{0 \leq j \leq 1} \bigg\vert \int_{0}^{1} {g(t)} \,{\rm d}W_t \bigg\vert,
	\end{displaymath}
	and by lemma $\ref{lemma 2.7}$, we have
	\begin{displaymath}
		\mathbb{P} \bigg( \sup\limits_{0 \leq j \leq 1} \bigg\vert \int_{0}^{1} {g(t)} \,{\rm d}W_t \bigg\vert \leq \epsilon \bigg) > 0.
	\end{displaymath}
	According to Lemma $\ref{lemma 2.2}$, $\mathcal{N}_H(\cdot) = \lim\limits_{n \to \infty} \mathcal{N}_n(\cdot)$ is a measurable norm.
\end{proof}

\section{Onsager-Machlup functional for 1-dimensional SDEs}
In this section, we focus on the $1$-dimensional case, thus set $d = 1$ and consider $W$ to be a $1$-dimensional Brownian motion. We proceed to present our principal findings related to the Onsager-Machlup functional for equation $\eqref{1.1}$.
\begin{theorem}	\label{theorem 3.1}
	Assume that $X_t$ is a solution of equation $\eqref{1.1}$, the reference path $\varphi$ is a function such that $\varphi_t - x$ belongs to Cameron-Martin space $H^1$ and the following conditions hold:
	\begin{itemize}
		\item[(1)] $f \in C^2_b\left([0, 1] \times \mathbb{R}, \mathbb{R} \right)$ and $f$ is globally Lipschitz continuous with constant $L$ for the second variable;
		\item[(2)] $g \in C^2_b\left([0, 1], \mathbb{R} \right)$ and there exist $ M > m > 0$, such that $m \leq g(t) \leq M$ for any $ t \in [0,1]$.
	\end{itemize} 
	Then, if we use the Hölder norm $\Vert \cdot \Vert_{\alpha}$ with $ 0< \alpha < \frac{1}{4}$, the Onsager-Machlup functional of $X_t$ exists and is given by
	\begin{equation}\label{2}
		OM(\varphi, \dot{\varphi}) = \int_{0}^{1} {\left( \frac{\dot{\varphi}_t - f(t, \varphi_t)}{g(t)} \right)^2 } \,{\rm d}t + \int_{0}^{1} {\partial_x f(t ,\varphi_t)} \,{\rm d}t,
	\end{equation}
	where $\dot{\varphi}:=\frac{\mathrm{d} \varphi(t)}{\mathrm{d} t}$, $\partial_x f(t, \varphi_t) := \frac{\partial f(t, x)}{\partial x}|_{\varphi_t}$.
\end{theorem}
\begin{proof}
	Let $Y_t$ be the solution for following stochastic integral equation:
	\begin{displaymath}
		Y_t = \varphi_t + \int_{0}^{t} {g(s)} \,{\rm d}W_s,
	\end{displaymath}
	where $\varphi_t \in H^1_{x_0}$. For convenience, denote $A_t := \int_{0}^{t} {g(s)} \,{\rm d}W_s$. Firstly, let's consider the relationship between ${\mathbb{P}\left(\Vert A\Vert \leq \epsilon\right)}$ and ${\mathbb{P}\left(\Vert W\Vert \leq \epsilon\right)}$. Under condition $(2)$, by Lemma $\ref{lemma 2.7}$ we have
	\begin{displaymath}
		\lim\limits_{\epsilon \to 0}\frac{\mathbb{P} \left(\Vert W \Vert \leq \frac{\epsilon}{M} \right)}{\mathbb{P}\left(\Vert W\Vert \leq \epsilon\right)}
		\leq \lim\limits_{\epsilon \to 0}\frac{\mathbb{P}\left(\Vert A\Vert \leq \epsilon\right)}{\mathbb{P}\left(\Vert W\Vert \leq \epsilon\right)} 
		\leq \lim\limits_{\epsilon \to 0}\frac{\mathbb{P} \left(\Vert W \Vert \leq \frac{\epsilon}{m} \right) }{\mathbb{P}\left(\Vert W\Vert \leq \epsilon\right)}.
	\end{displaymath}
	Applying Theorem $\ref{theorem 2.6}$, we can show that the ratio of ${\mathbb{P}\left(\Vert A\Vert \leq \epsilon\right)}$ to ${\mathbb{P}\left(\Vert W\Vert \leq \epsilon\right)}$ is bounded. Define
	\begin{equation}\label{3}
		\lim\limits_{\epsilon \to 0} \frac{\mathbb{P}\left(\Vert A\Vert \leq \epsilon\right)}{\mathbb{P}\left(\Vert W\Vert \leq \epsilon\right)} = C_{\frac{A}{W}},
	\end{equation}
	where $C_{\frac{A}{W}}$ is a constant that is related to $g(t)$ and $\alpha$.
	
	Due to $f \in C^2_b([0, 1] \times \mathbb{R}, \mathbb{R})$, $g \in C^2_b([0, 1], \mathbb{R})$ and $ \varphi_t \in H^1$, the Novikov condition $\mathbb{E}^{\mathbb{P}}\left( \mathrm{exp}\left\{ \int_{0}^{1} {(f - \dot{\varphi})} \,{\rm d}t \right\} \right) < \infty$ is clearly satisfied. Girsanov theorem implies that $\tilde{W}_t = W_t - \int_{0}^{t} {\frac{f(s, Y_s) - \dot{\varphi}_s}{g(s)} } \,{\rm d}s$ is a $1$-dimensional Brownian motion under new probability $\tilde{\mathbb{P}}$, which defined by $\frac{\tilde{\mathbb{P}}}{\mathbb{P}} := \mathcal{R}$ with
	\begin{displaymath}
		\mathcal{R} := {\rm exp} \left\{ {\int_{0}^{1} {\frac{f(t, Y_t) - \dot{\varphi}_t}{g(t)} } \,{\rm d}W_t - \frac{1}{2}\int_{0}^{1} {(\frac{f(t, Y_t) - \dot{\varphi}_t}{g(t)})^2 } \,{\rm d}t} \right\}.
	\end{displaymath}
	So based on Lemmas $\ref{L2.7}$ and $\ref{L2.8}$, we have the following result:
	\begin{equation}\label{4}
		\begin{aligned}
			\frac{\mathbb{P}\left(\Vert X_t -\varphi_t \Vert \leq \epsilon\right)}{\mathbb{P}\left(\Vert W\Vert \leq \epsilon\right)} 
			& = \frac{\tilde{\mathbb{P}}\left(\Vert Y_t -\varphi_t\Vert \leq \epsilon\right)}{\mathbb{P}\left(\Vert W\Vert \leq \epsilon\right)} 
			= \frac{\mathbb{E} \left( \mathcal{R}\mathbb{I}_{\Vert A\Vert \leq \epsilon} \right)}{\mathbb{P}\left(\Vert W\Vert \leq \epsilon\right)}
			\\& = \frac{\mathbb{E} \left( \mathcal{R}\mathbb{I}_{\Vert A\Vert \leq \epsilon} \right)}{\mathbb{P}\left(\Vert A\Vert \leq \epsilon\right)} \times
			\frac{\mathbb{P}\left(\Vert A\Vert \leq \epsilon\right)}{\mathbb{P}\left(\Vert W\Vert \leq \epsilon\right)}
			\\& = \mathbb{E}\left( \mathcal{R} \big| \Vert A\Vert \leq \epsilon \right) \times \frac{\mathbb{P}\left(\Vert A\Vert \leq \epsilon\right)}{\mathbb{P}\left(\Vert W\Vert \leq \epsilon\right)}
			\\& = \mathbb{E} \bigg( {\rm exp}\left\{ {\int_{0}^{1} {\frac{f(t, Y_t) - \dot{\varphi}_t}{g(t)} } \,{\rm d}W_t - \frac{1}{2}\int_{0}^{1} {\left(\frac{f(t, Y_t) - \dot{\varphi}_t}{g(t)}\right)^2 } \,{\rm d}t}  \right\} 
			\\ & \qquad   \big| \Vert A\Vert \leq \epsilon  \bigg) \times
			\frac{\mathbb{P}\left(\Vert A\Vert \leq \epsilon\right)}{\mathbb{P}\left(\Vert W\Vert \leq \epsilon\right)}
			\\& =  \mathbb{E} \left( {\rm exp}\left\{ { \sum_{i - 1}^{4} B_i} \right\} \big| \Vert A\Vert \leq \epsilon \right)
			\\& \quad \times {\rm exp}\left\{ { -\frac{1}{2} \int_{0}^{1} {\left( \frac{\dot{\varphi}_t - f(t, \varphi_t)}{g(t)} \right)^2 } \,{\rm d}t} \right\} \times   
			\frac{\mathbb{P}\left(\Vert A\Vert \leq \epsilon\right)}{\mathbb{P}\left(\Vert W\Vert \leq \epsilon\right)},
		\end{aligned}
	\end{equation}
	where
	\begin{align*}
		B_1 &= \int_{0}^{1} {\frac{f(t, Y_t)}{g(t)}} \,{\rm d}W_t,
		\\ B_2 &= - \int_{0}^{1} {\frac{\dot{\varphi}_t}{g(t)} } \,{\rm d}W_t,
		\\ B_3 &= \frac{1}{2} \int_{0}^{1} {\left(\frac{f(t, \varphi_t)}{g(t)} \right)^2 } \,{\rm d}t
		- \frac{1}{2} \int_{0}^{1} {\left(\frac{f(t, Y_t) }{g(t)} \right)^2}  \,{\rm d}t,
		\\ B_4 &= \int_{0}^{1} { \frac{(f(t, Y_t) - f(t, \varphi_t)) \dot{\varphi}_t}{(g(t))^2}  } \,{\rm d}t.
	\end{align*}
	
	For the second term $B_2$, applying Theorem $\ref{theorem 2.4}$ to $f := -c\frac{\dot{\varphi}_t}{g(t)}$ and Lemma $\ref{lemma 2.8}$, we have
	\begin{equation}\label{5}
		\limsup\limits_{\epsilon \to 0} \mathbb{E}\left({\rm exp}\left\{ cB_2 \right\} \big|\Vert A \Vert < \epsilon \right) = 1
	\end{equation}
	for all $c \in \mathbb{R}$.
	
	For the third term $B_3$, 
	\begin{displaymath}
		\begin{aligned}
			B_3 &= \frac{1}{2} \int_{0}^{1} { \frac{\left(f(t, \varphi_t) \right)^2 - \left(f(t, Y_t)  \right)^2 }{\left( g(t) \right)^2}} \,{\rm d}t
			\\ &= \frac{1}{2} \int_{0}^{1} { \frac{\left(f(t, \varphi_t) - f(t, Y_t)  \right)^2 + 2 \left(f(t, \varphi_t) - f(t, Y_t)\right) f(t, Y_t) }{\left( g(t) \right)^2}} \,{\rm d}t
			\\ &\leq  \frac{1}{2} \int_{0}^{1} { \frac{\left(f(t, \varphi_t) - f(t, Y_t)  \right)^2 }{\left( g(t) \right)^2}} \,{\rm d}t +  \int_{0}^{1} { \frac{ \vert \left(f(t, \varphi_t) - f(t, Y_t)\right) f(t, Y_t)\vert }{\left( g(t) \right)^2}} \,{\rm d}t.
		\end{aligned}
	\end{displaymath}
	Using that $f$ is Lipschitz continuous for the second variable, we have
	\begin{displaymath}
		\vert f(t, Y_t) - f(t, \varphi_t) \vert = \vert f(t, \varphi_t + A_t) - f(t, \varphi_t) \Vert \leq L \Vert A_t \Vert.
	\end{displaymath}
	When $\Vert A\Vert \leq \frac{\epsilon}{C_{\frac{A}{W}}}$, we obtain
	\begin{equation}\label{6}
		\vert f(t, Y_t) - f(t, \varphi_t) \vert  \leq \frac{L }{C_{\frac{A}{W}}} \epsilon.
	\end{equation}
	Inequality $\eqref{6}$ and the boundedness of $f$ and $g$ imply that
	\begin{equation}\label{7}
		\limsup\limits_{\epsilon \to 0} \mathbb{E}\left({\rm exp}\left\{ cB_3 \right\} \big|\Vert A \Vert < \epsilon \right) = 1
	\end{equation}
	for all $c \in \mathbb{R}$.
	
	For the fourth term $B_4$, applying inequality $\eqref{6}$ and the boundedness of $\dot{\varphi}_t$, we have
	\begin{equation}\label{8}
		\limsup\limits_{\epsilon \to 0} \mathbb{E}\left({\rm exp}\left\{ cB_4 \right\} \big|\Vert A \Vert_{\alpha} < \epsilon \right) = 1
	\end{equation}
	for all $c \in \mathbb{R}$.
	
	For the first term $B_1$, applying Taylor expansion, we have
	\begin{displaymath}
		\begin{aligned}
			f(t ,Y_t) &= f(t ,\varphi_t + A_t) = f(t ,\varphi_t) + \frac{\partial f(t, x)}{\partial x}|_{\varphi_t} A_t + R_t
			\\&= f(t ,\varphi_t) + \partial_2 f(t, \varphi_t) A_t + R_t,
		\end{aligned}
	\end{displaymath}
	and when $\Vert A\Vert \leq \frac{\epsilon}{C_{\frac{A}{W}}}$, there exists a constant $k$ such that
	\begin{displaymath}
		\sup\limits_{t \in [0, 1]} \Vert R_t \Vert \leq k \epsilon^2.
	\end{displaymath}
	Hence $B_1$ can be written
	\begin{displaymath}
		\begin{aligned}
			B_1 &= \int_{0}^{1} {\frac{f(t, \varphi_t + A_t)}{g(t)}} \,{\rm d}W_t = \int_{0}^{1} {\frac{f(t ,\varphi_t) + \partial_x f(t ,\varphi_t)A_t + R_t}{g(t)}} \,{\rm d}W_t
			\\ &:= C_1 + C_2 + C_3.
		\end{aligned}
	\end{displaymath}
	Applying Theorem $\ref{theorem 2.4}$ to $f(t) = \frac{f(t ,\varphi_t) }{g(t)} $ and Lemma $\ref{lemma 2.8}$, we have
	\begin{equation}\label{9}
		\limsup\limits_{\epsilon \to 0} \mathbb{E}\left({\rm exp}\left\{ cC_1 \right\} \big|\Vert A \Vert_{\alpha} < \epsilon \right) = 1
	\end{equation}
	for all $c \in \mathbb{R}$.
	In order to apply Theorem $\ref{theorem 2.5}$ we will express the term $C_2$ as a double stochastic integral with respect to $W$. We have
	\begin{displaymath}
		\begin{aligned}
			C_2 &= \int_{0}^{1} {\frac{ \partial_x f(t ,\varphi_t)A_t }{g(t)}} \,{\rm d}W_t 
			= \int_{0}^{1} {\int_{0}^{t} {\frac{ \partial_x f(t ,\varphi_t)g(s) }{g(t)}} \,{\rm d}W_s} \,{\rm d}W_t
			\\ &= \int_{0}^{1} {\int_{0}^{1} {\frac{ \partial_x f(t ,\varphi_t)g(s) 1_{s \leq t} }{g(t)}} \,{\rm d}W_s} \,{\rm d}W_t.
		\end{aligned}
	\end{displaymath}
	Define
	\begin{displaymath}
		F(s,t) := \frac{ \partial_x f(t ,\varphi_t)g(s) 1_{s \leq t} }{g(t)}.
	\end{displaymath}
	Hence, $C_2 = \int_{0}^{1} {\int_{0}^{1} {\tilde{F}} \,{\rm d}W_s} \,{\rm d}W_t$, where $\tilde{F} := \frac{1}{2} (F + F^{*})$ is the symmetrization of $F$. The operator $K(\tilde{F})$ is nuclear, and its trace can be computed as follows:
	\begin{displaymath}
		Tr.\tilde{F} = Tr.{F} = \frac{1}{2} \int_{0}^{1} {{F}(t, t)} \,{\rm d}t = \frac{1}{2} \int_{0}^{1} {\partial_x f(t ,\varphi_t)} \,{\rm d}t.
	\end{displaymath}
	By Theorem $\ref{theorem 2.5}$ and Lemma $\ref{lemma 2.8}$, we have
	\begin{equation}\label{10}
		\limsup\limits_{\epsilon \to 0} \mathbb{E}\left({\rm exp}\left\{ cC_2 \right\} \big|\Vert A \Vert < \epsilon \right) = {\rm exp} \left( - \frac{1}{2} \int_{0}^{1} {\partial_x f(t ,\varphi_t)} \,{\rm d}t \right)
	\end{equation}
	for all $c \in \mathbb{R}$.
	Finally, we study the behaviour of the term $C_3$. For any $c \in \mathbb{R}$ and $\delta > 0$, we have
	\begin{equation}\label{11}
		\begin{aligned}
			\mathbb{E} \left( {\rm exp} \left\{ cC_3 \right\}\big| \Vert A \Vert \leq \epsilon \right) 
			&= \int_{0}^{\infty} {e^x \mathbb{P}\left( \left| c\int_{0}^{1} {R_t} \,{\rm d}W_t  \right| > x \big| \Vert A \Vert \leq \epsilon \right)} \,{\rm d}x
			\\ &\leq \int_{\delta}^{\infty} {e^x \mathbb{P}\left( \left| c\int_{0}^{1} {R_t} \,{\rm d}W_t \right| > x \big| \Vert A \Vert \leq \epsilon \right)} \,{\rm d}x 
			\\ &\quad + e^{\delta} \mathbb{P}\left( \left| c\int_{0}^{1} {R_t} \,{\rm d}W_t \right| > \delta \big| \Vert A \Vert \leq \epsilon \right).
		\end{aligned}
	\end{equation}
	Define the martingale $M_t = c\int_{0}^{t} {R_s} \,{\rm d}W_s$. We have estimate about its quadratic variation
	\begin{displaymath}
		\langle M_t \rangle = c^2\int_{0}^{t} {\Vert R_s \Vert^2} \,{\rm d}s \leq C \epsilon^4
	\end{displaymath}
	for some $C > 0$. Using the exponential inequality for martingales we have
	\begin{displaymath}
		\mathbb{P}\left( \left| c\int_{0}^{1} {R_t} \,{\rm d}W_t \right| > x, \Vert A \Vert \leq \epsilon \right) \leq {\rm exp}\left\{ -\frac{x^2}{2c\epsilon^4} \right\}.
	\end{displaymath}
	Then, by Lemma $\ref{lemma 2.7}$, we have
	\begin{displaymath}
		\mathbb{P}\left( \left| c\int_{0}^{1} {R_t} \,{\rm d}W_t \right| > x \big| \Vert A \Vert \leq \epsilon \right) \leq {\rm exp}\left\{ -\frac{x^2}{2c\epsilon^4} \right\}
		{\rm exp}\left\{ c \left( \frac{\epsilon}{M} \right)^{-\frac{2}{1 - 2\alpha}} \right\}.
	\end{displaymath}
	Applying the latter estimate and taking the limit in equation \(\eqref{11}\), we obtain that for \( 0 < \alpha < \frac{1}{4} \), the inequality \( \frac{2}{1 - 2\alpha} < 4 \) holds. Therefore, we conclude that
	\begin{equation}\label{12}
		\limsup\limits_{\epsilon \to 0} \mathbb{E}\left({\rm exp}\left\{ cC_3 \right\} \big| \Vert A \Vert \leq \epsilon \right) = 1
	\end{equation}
	for all $c \in \mathbb{R}$, as $\epsilon \to 0$ and $\delta \to 0$.
	
	Finally, by Lemma $\ref{lemma 2.3}$ and formulas $\eqref{3}$ - $\eqref{5}$, $\eqref{7}$ - $\eqref{10}$ and $\eqref{12}$, we have
	\begin{displaymath}
		\begin{aligned}
			\lim\limits_{\epsilon \to 0} \frac{\mathbb{P}\left(\Vert X_t -\varphi_t \Vert \leq \epsilon\right)}{\mathbb{P}\left(\Vert W\Vert \leq \epsilon\right)} = C_{\frac{A}{W}} {\rm exp} \left\{ -\frac{1}{2} \int_{0}^{1} {\left( \frac{\dot{\varphi}_t - f(t, \varphi_t)}{g(t)} \right)^2 } \,{\rm d}t - \frac{1}{2} \int_{0}^{1} {\partial_x f(t ,\varphi_t)} \,{\rm d}s \right\}.	
		\end{aligned}\notag
	\end{displaymath}
Consequently,
	\begin{displaymath}
		OM(\varphi, \dot{\varphi}) = \int_{0}^{1} {\left( \frac{\dot{\varphi}_t - f(t, \varphi_t)}{g(t)} \right)^2 } \,{\rm d}t + \int_{0}^{1} {\partial_x f(t ,\varphi_t)} \,{\rm d}t.
	\end{displaymath}
\end{proof}

\section{Onsager-Machlup functional for $n$-dimensional SDEs}
In this section, our exploration extends to the $n$-dimensional scenario, with $W_t := (W^1_t, W^2_t,..., W^n_t)^T$ representing an $n$-dimensional Brownian motion. For functions $f \in C^2_b\left(D, \mathbb{R}^n \right), g \in C^2_b\left(D, \mathbb{R}^{n \times n} \right)$, we introduce the following definitions for spaces and norms:
\begin{equation}
	\begin{aligned}
		f &:= \sum\limits_{i = 1}^{n} f^{i} e_i,
		\quad \Vert f \Vert = \frac{1}{n} \sum\limits_{i = 1}^{n} \Vert f^{i} \Vert,
		\\ g &:= \sum\limits_{i,j = 1}^{n} g^{(i,j)} e_i \otimes e_j,
		\quad \Vert g \Vert = \frac{1}{n^2} \sum\limits_{i,j = 1}^{n} \Vert g^{(i,j)} \Vert,
	\end{aligned}\notag
\end{equation}
where $f^{i}, g^{(i,j)} \in  C^2_b\left(D, \mathbb{R} \right)$, $\left\{ e_i \right\}_{1 \leq i \leq n}$ is a set of orthonormal basis in space $\mathbb{R}^n$ and $\otimes$ denotes the tensor product. Additionally, we assume that $\langle X_1, X_2 \rangle$ represents the inner product between vectors $X_1$ and $X_2$, while $ M \cdot X$ represents the product between matrices $M$ and vectors (or matrices) $X$.

To elaborate on the Onsager-Machlup functional for equation $\eqref{1.1}$, it becomes necessary to further define $f$ and $g$.

\begin{definition} We define some conditions regarding $f$ and $g$:
	\begin{itemize}
		\item[(C1)]$f \in C^2_b\left(\left[0, 1 \right] \times \mathbb{R}^n, \mathbb{R}^n \right)$ and $f$ is globally Lipschitz continuous with constant $L$ for the second variable;
		\item[(C2)] $g \in C^2_b\left(\left[0, 1 \right], \mathbb{R}^{n \times n} \right)$, and $g$ is positive and bounded, that is, there existing $ M > m > 0$, such that $m \leq \mathrm{det} \left( g \left( t \right) \right) \leq M$ for any $ t \in [0,1]$;
		\item[(C3)] $g(t)^{-1}$, generalized inverse matrix  of $g(t)$, is continuous for any $ t \in [0,1]$.
	\end{itemize}
\end{definition}
\begin{definition} 
	Assumming $g : \left[0, 1 \right] \to \mathbb{R}^{n \times n}$ is nondegenerate, that is, $\mathrm{det} \left( g \left( t \right) \right) \neq 0$ for any $ t \in \left[0,1 \right] $, we say that $h : \left[0, 1 \right] \to \mathbb{R}^{n \times n}$ is the generalized inverse matrix of the function matrix $g$, if $h \cdot g = g \cdot h = I^{n \times n}$ for any $ t \in \left[0,1 \right]$. We will denote the generalized inverse matrix $h$ as $g^{-1}$.
\end{definition}
\begin{theorem}	\label{theorem 4.3}
	Assume that $X_t$ is a solution of equation $\eqref{1.1}$, the reference path $\varphi$ is a function such that $\varphi_t - x$ belongs to Cameron-Martin $H^1$. And assume that $f$ and $g$ satisfy continuous $(C1)$, $(C2)$ and $(C3)$. If we use the Hölder norm $\Vert \cdot \Vert_{\alpha}$ with $ 0< \alpha < \frac{1}{4}$, then the Onsager-Machlup functional of $X_t$ exists and has the form
	\begin{equation}\label{13}
		OM(\varphi, \dot{\varphi}) = \int_{0}^{1} \left| {g(t)}^{-1} \cdot \left({\dot{\varphi}_t - f(t, \varphi_t)} \right)\right|^2 \,{\rm d}t + \int_{0}^{1} {{\rm div}^{g}_{x} f(t ,\varphi_t)} \,{\rm d}t,
	\end{equation}
	where $\dot{\varphi}:=\frac{\mathrm{d} \varphi(t)}{\mathrm{d} t}$, ${\rm div}^{g}_{x} f(t, \varphi_t) := Tr.\left({g(t)^{-1}}{\nabla f(t ,\varphi_t)}g(t) \right)$.
\end{theorem}
\begin{proof}
	Similar to the proof method of Theorem $\ref{theorem 3.1}$, let $Y_t$ be the solution for following stochastic integral equation:
	\begin{displaymath}
		Y_t = \varphi_t + \int_{0}^{t} {g(s)} \cdot \,{\rm d}W_s,
	\end{displaymath}
	where $\varphi_t \in H^1$. For convenience, denote $A_t := \int_{0}^{t} {g \left( s \right)} \cdot \,{\rm d}W_s$. In the $n$-dimensional case, we can verify that $\eqref{3}$ still holds.
	Due to $f \in C^2_b \left(\left[0, 1 \right] \times \mathbb{R}^n, \mathbb{R}^n \right)$, $g \in C^2_b \left(\left[0, 1 \right], \mathbb{R}^{n \times n} \right)$ and $ \varphi_t \in H^1$, the Novikov condition $\mathbb{E}^{\mathbb{P}}\left( \mathrm{exp}\left\{ \int_{0}^{1} {(f - \dot{\varphi})} \,{\rm d}t \right\} \right) < \infty$ is clearly satisfied. Girsanov theorem implies that $\tilde{W}_t = W_t - \int_{0}^{t} {{g(s)^{-1}} \cdot {\left(f \left(s, Y_s \right) - \dot{\varphi}_s \right)} } \,{\rm d}s$ is an $n$-dimensional Brownian motion under new probability $\tilde{\mathbb{P}}$, which defined by $\frac{\tilde{\mathbb{P}}}{\mathbb{P}} := \mathcal{R}$ with
	\begin{displaymath}
		\mathcal{R} := {\rm exp} \left\{ {\int_{0}^{1} \big\langle  {{g(t)^{-1}} \cdot {\left(f \left(t, Y_t \right) - \dot{\varphi}_t \right)} }, \,{\rm d}W_t \big\rangle - \frac{1}{2}\int_{0}^{1} {\left|{g \left( s \right)^{-1}}\cdot{\left( f \left(t, Y_t \right) - \dot{\varphi}_t \right)} \right|^2 } \,{\rm d}t} \right\}.
	\end{displaymath}
	So we have
	\begin{equation} \label{14}
		\begin{aligned}
			\frac{\mathbb{P}\left(\Vert X_t -\varphi_t \Vert \leq \epsilon\right)}{\mathbb{P}\left(\Vert W\Vert \leq \epsilon\right)} 
			& = \mathbb{E} \bigg( {\rm exp}  \left\{ {\int_{0}^{1} \big\langle  {{g(t)^{-1}} \cdot {\left(f \left(t, Y_t \right) - \dot{\varphi}_t \right)} }, \,{\rm d}W_t \big\rangle - \frac{1}{2}\int_{0}^{1} {\left| {g \left( s \right)^{-1}}\cdot{\left( f \left(t, Y_t \right) - \dot{\varphi}_t \right)} \right|^2 } \,{\rm d}t} \right\} 
			\\ & \qquad  \big| \Vert A\Vert \leq \epsilon \bigg) \times  
			\frac{\mathbb{P}\left(\Vert A\Vert \leq \epsilon\right)}{\mathbb{P}\left(\Vert W\Vert \leq \epsilon\right)}
			\\& =  \mathbb{E} \left( {\rm exp}\left\{ { \sum_{i - 1}^{4} B_i} \right\} \big| \Vert A\Vert \leq \epsilon \right) \cdot {\rm exp}\left\{ { -\frac{1}{2} \int_{0}^{1} {\left| {g(t)^{-1}}\cdot{\left(\dot{\varphi}_t - f(t, \varphi_t)\right)} \right|^2 } \,{\rm d}t} \right\} \times   
			\frac{\mathbb{P}\left(\Vert A\Vert \leq \epsilon\right)}{\mathbb{P}\left(\Vert W\Vert \leq \epsilon\right)}.
		\end{aligned}
	\end{equation}
	where
	\begin{align*}
		B_1 &= \int_{0}^{1} \big\langle {{g(t)^{-1}}\cdot{f(t, Y_t)}}, \,{\rm d}W_t \big\rangle,
		\\ B_2 &= - \int_{0}^{1} \big\langle {{g(t)^{-1}}\cdot{\dot{\varphi}_t} }, \,{\rm d}W_t \big\rangle,
		\\ B_3 &= \frac{1}{2} \int_{0}^{1} {\left|{g(t)^{-1}}\cdot {f(t, \varphi_t)} \right|^2 } \,{\rm d}t
		- \frac{1}{2} \int_{0}^{1} {\left| {g(t)^{-1}}\cdot {f(t, Y_t) } \right|^2}  \,{\rm d}t,
		\\ B_4 &= \int_{0}^{1} {\big\langle  {g(t)^{-2}}\cdot {(f(t, Y_t) - f(t, \varphi_t)), \dot{\varphi}_t} \big\rangle } \,{\rm d}t.
	\end{align*}
	
	Under conditions $(C2), (C3)$, we can show that $g(t)^{-1}$ is a positive, bounded and continuous function matrix. Note that $\dot{\varphi}_t \in L^2([0,1], \mathbb{R}^n)$. Let $\phi_t := {g(t)^{-1}}\cdot{\dot{\varphi}_t}$, we have $\phi_t := \left( \phi^1_t, \phi^2_t, ... , \phi^n_t  \right)^{T}  \in L^2([0,1], \mathbb{R}^n)$.
	
	For the second term $B_2$, we have
	\begin{equation}
		\begin{aligned}
			B_2 &= - \int_{0}^{1} \big\langle {{g(t)^{-1}}\cdot{\dot{\varphi}_t} }, \,{\rm d}W_t \big\rangle
			&= - \int_{0}^{1} \big\langle {\phi_t }, \,{\rm d}W_t \big\rangle
			&= - \sum_{i - 1}^{n}\int_{0}^{1}  {\phi^i_t } \,{\rm d}W^i_t.
		\end{aligned}\notag
	\end{equation}
	Applying Theorem $\ref{theorem 2.4}$ to $h = -c\phi^i_t$ and Lemma $\ref{lemma 2.3}$, Lemma $\ref{lemma 2.8}$, we obtain
	\begin{equation}\label{15}
		\limsup\limits_{\epsilon \to 0} \mathbb{E}\left({\rm exp}\left\{ cB_2 \right\} \big|\Vert A \Vert < \epsilon \right) = 1
	\end{equation}
	for all $c \in \mathbb{R}$.
	
	For the third term $B_3$, 
	\begin{displaymath}
		\begin{aligned}
			B_3 &= \frac{1}{2} \int_{0}^{1} {\left|{g(t)^{-1}}\cdot {f(t, \varphi_t)} \right|^2 } - {\left| {g(t)^{-1}}\cdot {f(t, Y_t) } \right|^2}  \,{\rm d}t
			\\ &\leq \frac{1}{2} \int_{0}^{1} { {\left| {g(t)^{-1}}\cdot \left(f(t, \varphi_t) - f(t, Y_t) \right) \right|^2 + 2 \left| {g(t)^{-1}}\cdot \left( f(t, \varphi_t) - f(t, Y_t) \right) \right| \left|{g(t)^{-1}}\cdot f(t, Y_t) \right|}} \,{\rm d}t
			\\ &\leq \frac{1}{2} \int_{0}^{1} { {\left| {g(t)^{-1}}\cdot \left(f(t, \varphi_t) - f(t, Y_t) \right) \right|^2 \,{\rm d}t +  \int_{0}^{1} \left| {g(t)^{-1}}\cdot \left( f(t, \varphi_t) - f(t, Y_t) \right) \right| \left|{g(t)^{-1}}\cdot f(t, Y_t) \right|}} \,{\rm d}t.
		\end{aligned}
	\end{displaymath}
	Using that $f$ is Lipschitz continuous for the second variable, we have
	\begin{displaymath}
		\vert f(t, Y_t) - f(t, \varphi_t) \vert = \vert f(t, \varphi_t + A_t) - f(t, \varphi_t) \Vert \leq L \Vert A_t \Vert.
	\end{displaymath}
	When $\Vert A\Vert \leq \frac{\epsilon}{C_{\frac{A}{W}}}$, we obtain
	\begin{equation}\label{16}
		\vert f(t, Y_t) - f(t, \varphi_t) \vert  \leq \frac{L }{C_{\frac{A}{W}}} \epsilon.
	\end{equation}
	Inequality $\eqref{16}$ and the boundedness of $f$ and $g^{-1}$ imply that
	\begin{equation}\label{17}
		\limsup\limits_{\epsilon \to 0} \mathbb{E}\left({\rm exp}\left\{ cB_3 \right\} \big|\Vert A \Vert < \epsilon \right) = 1
	\end{equation}
	for all $c \in \mathbb{R}$.
	
	For the fourth term $B_4$, applying inequality $\eqref{16}$ and the boundedness of $\dot{\varphi}_t$ and $g^{-1}$, we have
	\begin{equation}\label{18}
		\limsup\limits_{\epsilon \to 0} \mathbb{E}\left({\rm exp}\left\{ cB_4 \right\} \big|\Vert A \Vert_{\alpha} < \epsilon \right) = 1
	\end{equation}
	for all $c \in \mathbb{R}$.
	
	For the first term $B_1$, applying Taylor expansion, we have
	\begin{displaymath}
		\begin{aligned}
			f(t ,Y_t) &= f(t ,\varphi_t + A_t) = f(t ,\varphi_t) + \nabla f(t, x) |_{\varphi_t} A_t + R_t
			\\&= f(t ,\varphi_t) + \nabla f(t, \varphi_t) A_t + R_t,
		\end{aligned}
	\end{displaymath}
	where $\nabla$ is gradient operator. For $\Vert A\Vert \leq \epsilon$, we have
	\begin{displaymath}
		\sup\limits_{t \in [0, 1]} \Vert R_t \Vert \leq k \epsilon^2.
	\end{displaymath}
	Hence, $B_1$ can be written
	\begin{displaymath}
		\begin{aligned}
			B_1 &= \int_{0}^{1} \big\langle {{g(t)^{-1}}\cdot{f(t, Y_t)}}, \,{\rm d}W_t \big\rangle = \int_{0}^{1} \big\langle {{g(t)^{-1}}\cdot{\left(f(t ,\varphi_t) + \nabla f(t, \varphi_t)\cdot A_t + R_t \right)}}, \,{\rm d}W_t \big\rangle
			\\ &:= B_{11} + B_{12} + B_{13}.
		\end{aligned}
	\end{displaymath}
	The term $B_{11}$ has the same expression as $B_2$
	\begin{displaymath}
		\begin{aligned}
			B_{11} = \int_{0}^{1} \big\langle {{g(t)^{-1}}\cdot{f(t ,\varphi_t)}}, \,{\rm d}W_t \big\rangle.
		\end{aligned}
	\end{displaymath}
	Due to $f \in C^2_b \left(\left[0, 1 \right] \times \mathbb{R}^n, \mathbb{R}^n \right)$, we can show that ${g(t)^{-1}}\cdot{f(t ,\varphi_t)}  \in L^2([0,1], \mathbb{R}^n)$. Using the same method as item $B_2$ yields
	\begin{equation}\label{19}
		\limsup\limits_{\epsilon \to 0} \mathbb{E}\left({\rm exp}\left\{ cB_{11} \right\} \big|\Vert A \Vert < \epsilon \right) = 1
	\end{equation}
	for all $c \in \mathbb{R}$.
	In order to apply Theorem $\ref{theorem 2.5}$, we will express the term $B_{12}$ as a double stochastic integral with respect to $W$. We have
	\begin{displaymath}
		\begin{aligned}
			B_{12} &= \int_{0}^{1} \langle{{g(t)^{-1}}\cdot{ \nabla f(t ,\varphi_t)\cdot A_t }}, \,{\rm d}W_t \rangle
			= \int_{0}^{1} {\int_{0}^{t} {{g(t)^{-1}}\cdot{ \nabla f(t ,\varphi_t)\cdot g(s) }}  \,{\rm d}W_s} \otimes \,{\rm d}W_t
			\\ &= \int_{0}^{1} {\int_{0}^{1} {{g(t)^{-1}}\cdot{ \nabla f(t ,\varphi_t)\cdot g(s) 1_{s \leq t} }} \,{\rm d}W_s} \otimes \,{\rm d}W_t.
		\end{aligned}
	\end{displaymath}
	Define
	\begin{displaymath}
		F(s,t) := {g(t)^{-1}}\cdot{ \nabla f(t ,\varphi_t)\cdot g(s) 1_{s \leq t} }.
	\end{displaymath}
	Hence, $B_{12} = \int_{0}^{1} \int_{0}^{1} {\tilde{F}} \,{\rm d}W_s \otimes \,{\rm d}W_t$, where $\tilde{F} := \frac{1}{2} (F + F^{*})$ is the symmetrization of $F$. According to conditions $(C2)$ and $(C3)$, the operator $K(\tilde{F})$ is nuclear, and its trace can be computed as follows:
	\begin{displaymath}
		Tr.\tilde{F} = Tr.{F} = \frac{1}{2} \int_{0}^{1} {{F}(t, t)} \,{\rm d}t = \frac{1}{2}  \int_{0}^{1} Tr.({g(t)^{-1}}\cdot{\nabla f(t ,\varphi_t)}\cdot g(t)) \,{\rm d}t.
	\end{displaymath}
	By Theorem $\ref{theorem 2.5}$ and Lemma $\ref{lemma 2.8}$, we have
	\begin{equation}\label{20}
		\limsup\limits_{\epsilon \to 0} \mathbb{E}\left({\rm exp}\left\{ cB_{12} \right\} \big|\Vert A \Vert < \epsilon \right) = {\rm exp} \left( - \frac{1}{2} \int_{0}^{1} Tr.\left({g(t)^{-1}}\cdot{\nabla f(t ,\varphi_t)}\cdot g(t) \right) \,{\rm d}t \right)
	\end{equation}
	for all $c \in \mathbb{R}$.
	Finally, we study the behaviour of the term $B_{13}$. For any $c \in \mathbb{R}$ and $\delta > 0$, we have
	\begin{equation}\label{21}
		\begin{aligned}
			\mathbb{E} \left( {\rm exp} \left\{ cB_{13} \right\} \big| \Vert A \Vert \leq \epsilon \right) 
			&= \int_{0}^{\infty} {e^x \mathbb{P}\left( \left| c\int_{0}^{1} \langle {g(t)^{-1}} \cdot {R_t}, \,{\rm d}W_t \rangle \right| > x \big| \Vert A \Vert \leq \epsilon \right)} \,{\rm d}x
			\\ &\leq \int_{\delta}^{\infty} {e^x \mathbb{P}\left( \left| c\int_{0}^{1} \langle {g(t)^{-1}} \cdot {R_t}, \,{\rm d}W_t \rangle \right| > x \big| \Vert A \Vert \leq \epsilon \right)} \,{\rm d}x 
			\\ &\quad + e^{\delta} \mathbb{P}\left( \left| c\int_{0}^{1} \langle {g(t)^{-1}} \cdot {R_t}, \,{\rm d}W_t \rangle \right| > \delta \big| \Vert A \Vert \leq \epsilon \right).
		\end{aligned}
	\end{equation}
	Define the martingale $M_t = c\int_{0}^{t} \langle {g(t)^{-1}} \cdot {R_s}, \,{\rm d}W_s \rangle$. We have estimate about its quadratic variation
	\begin{displaymath}
		\langle M_t \rangle = c^2\int_{0}^{t} {\Vert {g(t)^{-1}} \cdot R_s \Vert^2} \,{\rm d}s \leq C \epsilon^4
	\end{displaymath}
	for some $C > 0$. Using the exponential inequality for martingales, we obtain
	\begin{displaymath}
		\mathbb{P}\left( \left| c\int_{0}^{1} \langle {g(t)^{-1}} \cdot {R_t}, \,{\rm d}W_t \rangle \right| > x, \Vert A \Vert \leq \epsilon \right) \leq {\rm exp}\left\{ -\frac{x^2}{2c\epsilon^4} \right\}.
	\end{displaymath}
	Then, by Lemma $\ref{lemma 2.7}$, we have
	\begin{displaymath}
		\mathbb{P}\left( \left| c\int_{0}^{1} \langle {g(t)^{-1}} \cdot {R_t}, \,{\rm d}W_t \rangle \right| > x \big| \Vert A \Vert \leq \epsilon \right) \leq {\rm exp}\left\{ -\frac{x^2}{2c\epsilon^4} \right\}
		{\rm exp}\left\{ c \left( \frac{\epsilon}{M} \right)^{-\frac{2}{1 - 2\alpha}} \right\}.
	\end{displaymath}
	Applying the latter estimate and taking limits in $\eqref{21}$, if $0 < \alpha < \frac{1}{4}$ yields
	\begin{equation}\label{22}
		\limsup\limits_{\epsilon \to 0} \mathbb{E}\left({\rm exp}\left\{ cB_{13} \right\} \big| \Vert A \Vert \leq \epsilon \right) = 1
	\end{equation}
	for all $c \in \mathbb{R}$, as $\epsilon \to 0$ and $\delta \to 0$.
	
	Finally, by Lemma $\ref{lemma 2.3}$ and inequalities $\eqref{3}$, $\eqref{15}$, $\eqref{17}$ - $\eqref{20}$ and $\eqref{22}$, we have
	\begin{displaymath}
		\begin{aligned}
			\lim\limits_{\epsilon \to 0} \frac{\mathbb{P}\left(\Vert X_t -\varphi_t \Vert \leq \epsilon\right)}{\mathbb{P}\left(\Vert W\Vert \leq \epsilon\right)} = C_{\frac{A}{W}} {\rm exp} \left\{  -\frac{1}{2} \int_{0}^{1} \left| 	{g(t)}^{-1} \cdot  \left( {\dot{\varphi}_t - f(t, \varphi_t)} \right) \right|^2 \,{\rm d}t - \frac{1}{2} \int_{0}^{1} {{\rm div}^{g}_{x} f(t ,\varphi_t)} \,{\rm d}t \right\}.	
		\end{aligned}\notag
	\end{displaymath}
	Consequently,
	\begin{displaymath}
		OM(\varphi, \dot{\varphi}) =  \int_{0}^{1} \left| 	{g(t)}^{-1} \cdot  \left( {\dot{\varphi}_t - f(t, \varphi_t)} \right) \right|^2 \,{\rm d}t + \int_{0}^{1} {{\rm div}^{g}_{x} f(t ,\varphi_t)} \,{\rm d}t.
	\end{displaymath}
	\end{proof}

\section{Numerical experiments}
In this section, 
drawing upon the foundational insights from Theorem $\ref{theorem 4.3}$ and Theorem $\ref{theorem 3.1}$ of this paper, we illustrate our methodology through examples and numerical simulations for a one-dimensional SDE and a fast-slow SDE system. Specifically, we first calculate the respective equations' Onsager-Machlup functionals. Then, by minimizing the OM functionals, we determine the most probable pathway between two metastable states. Finally, computer simulations are employed to generate visual representations, validating the accuracy of our theoretical findings.
\begin{example}\label{example 5.1}
	Consider the following $1$-dimensional SDE:
	\begin{equation}\label{5.1}
		\begin{aligned}
			\mathrm{d} x(t) = t(4x(t) - (x(t))^3)\mathrm{d} t + (t + 1)\mathrm{d}W_t,
		\end{aligned}
	\end{equation}
where $x \in \mathbb{R}$, $t \in [0,1]$. We define $f(t,x) = t(4x(t) - (x(t))^3)$ and $g(t) = t + 1$. It is demonstrable that the functions $f$ and $g$ meet the criteria stipulated by Theorem $\ref{theorem 3.1}$. Given $f$'s properties, the system's metastable states are identified at $x = -2$ and $x = 2$, allowing us to calculate the most probable path between these states.
\end{example}
By Theorem $\ref{theorem 3.1}$ we can obtain the Onsager-Machlup functional for $\eqref{5.1}$:
\begin{align*}
	OM(y, \dot{y}) = \int_{0}^{1} {\left( \frac{\dot{y}_t - t(4y(t) - (y(t))^3)}{t + 1} \right)^2 } \,{\rm d}t + \int_{0}^{1} {t(4 - 3(y(t))^2)} \,{\rm d}t.
\end{align*}
Then we can find the most probable path $y$ for $x$ by minimizing the corresponding Onsager-Machlup functional $OM(y, \dot{y})$ with the help of variational principle. Firsy, we define the Lagrangian $L(t, y, \dot{y})$:
\begin{align*}
	L(t, y, \dot{y}) = {\left( \frac{\dot{y}_t - t(4y(t) - (y(t))^3)}{t + 1} \right)^2 } + {t(4 - 3(y(t))^2)}.
\end{align*}
Then, we can obtain the Euler Lagrange equation,
\begin{align}\label{5.2}
	\frac{\partial L}{\partial y} - \frac{\partial}{\partial t}\left(\frac{\partial L}{\partial \dot{y}}\right)= 0,
\end{align}
where
\begin{align}
	\frac{\partial L}{\partial y} &= \frac{-2t(\dot{y} - t(4y - y^3)) * (4 - 3y^2))}{(t + 1)^2} - 6ty, \label{5.3}\\
	\frac{\partial}{\partial t}\left(\frac{\partial L}{\partial \dot{y}}\right) &= \frac{2}{(t + 1)^3} \left(\left( \ddot{y} -\left(4y - y^3\right) -t \left(4\dot{y} -3y^2\dot{y}\right) \right) \left(1+ t\right) - 2 \left(\dot{y} -t\left( 4y - y^3\right)\right)\right).\label{5.4}
\end{align}
By bringing $\eqref{5.3}$ and $\eqref{5.4}$ into $\eqref{5.2}$, we can obtain
\begin{align}\label{5.5}
	\ddot{y} = (4y - y^3) + t(4\dot{y} - 3y^2\dot{y}) + \frac{(2(\dot{y} - t(4y - y^3)))}{1 + t} - t(\dot{y} - t(4y - y^3))(4 - 3y^2) - 3ty(1 + t)^2.
\end{align}
Under conditions $x(0)=-2$ and $x(1)=2$, according to equation $\eqref{5.5}$, we obtain the most probable path of system $\eqref{5.1}$ between two different metastable states through numerical simulation (see $Fig.\ref{F.1}$).

\begin{example}
	The significance of stochastic volatility in the domains of option pricing and asset risk management is well-established. Jean-Pierre Fouque et al. \cite{21} have introduced a spectrum of multiscale stochastic volatility models characterized by dual diffusions operating on distinct temporal scales: one exhibiting rapid fluctuations and the other manifesting more gradual changes. This approach has been demonstrated to enhance model fitting accuracy substantially. In this framework, the price of the underlying asset, denoted by $X(t)$, is represented as the solution to a SDE:
	\begin{align*}
		\mathrm{d}X(t) = AX(t)\mathrm{d}t +  B(Y(t), Z(t))X(t)\mathrm{d}W_t,
	\end{align*}
	where the volatility process $B(Y(t), Z(t))$ is driven by two diffusion processes fast scale volatility factor $Y(t)$ and slow scale volatility factor $Z(t)$:
	\begin{align*}
		\mathrm{d}Y(t) &= \frac{1}{\epsilon} F_1(t, Y(t))\mathrm{d}t + \frac{1}{\sqrt{\epsilon}} G_1(t)\mathrm{d}W^1_t,\\
		\mathrm{d}Z(t) &= \delta AF_2(t, Z(t))\mathrm{d}t + \sqrt{\delta} G_2(t)\mathrm{d}W^2_t,
	\end{align*}
	where $\frac{1}{\epsilon}$ signifies the parameter governing the fast scale, while $\delta$ denotes the parameter for the slow scale. It can be demonstrated that the Onsager-Machlup functional effectively identifies the most probable path for multiscale factor random volatility. Moreover, our analysis extends to more intricate scenarios where factors influencing fast scale fluctuations are interlinked with those affecting slow scale fluctuations. To illustrate this, we present an example involving a $2$-dimensional stochastic differential equation:
	\begin{equation}\label{5.6}
		\begin{aligned}
			\mathrm{d} \left( \begin{matrix}
				x_1(t)\\
				x_2(t)
			\end{matrix}\right) 
		= 0.04 \cdot\left( \begin{matrix}
			a^2x_1(t)(8 - x_1(t)x_2(t) - (x_1(t))^2)\\
			b^2x_2(t)(8 - x_1(t)x_2(t) - (x_2(t))^2)
		\end{matrix}\right)\cdot \mathrm{d} t 
		+ 0.4\left( \begin{matrix}
			a(1 &+ t) \quad &0 \\
			&0 \quad b(2 &+ t)
		\end{matrix}\right)\cdot
		\left( \begin{matrix}
			\mathrm{d}W^1_t \\
			\mathrm{d}W^2_t
		\end{matrix}\right),
		\end{aligned}
	\end{equation}
	where $x(t) = \left( \begin{matrix}
		x_1(t)\\
		x_2(t)
	\end{matrix}\right) \in \mathbb{R}^2$, $t \in [0,1]$ and $a^2$ is fast scale, $b^2$ is slow scale. We set $g(t) = 0.4\left( \begin{matrix}
	a(1 &+ t) \quad &0 \\
	&0 \quad b(2 &+ t)
\end{matrix}\right)$, $f(t,x) = 0.04 \cdot\left( \begin{matrix}
a^2 x_1(t)(8 - x_1(t)x_2(t) - (x_1(t))^2)\\
b^2 x_2(t)(8 - x_1(t)x_2(t) - (x_2(t))^2)
\end{matrix}\right) \mathrm{d} t$, and $W^1_t$ and $W^2_t$ are independent $1$-dimensional Brownian motions. It is demonstrable that the functions $f$ and $g$ meet the criteria stipulated by Theorem $\ref{theorem 4.3}$. Given $f$'s properties, the system's metastable states are identified at $(-2,-2)$ and $(2,2)$, allowing us to calculate the most probable path between these states.
\end{example}
By Theorem $\ref{theorem 4.3}$ we can obtain the Onsager-Machlup functional for $\eqref{5.6}$:
	\begin{align*}
		OM(y, \dot{y}) &= \int_{0}^{1} \left| {g(t)}^{-1} \cdot \left({\dot{\varphi}_t - f(t, \varphi_t)} \right)\right|^2 \,{\rm d}t + \int_{0}^{1} {{\rm div}^{g}_{x} f(t ,\varphi_t)} \,{\rm d}t\\
		&= 6.25\left.\left.\left\|\begin{bmatrix}a(1+t)&0\\0&b(2+t)\end{bmatrix}\right.^{-1}\cdot\left(\begin{bmatrix}\dot{y}_1\\\dot{y}_2\end{bmatrix}\right.-\begin{bmatrix}0.04(-a^2 y_1(y_1^2+y_1 y_2-8))\\0.04(-b^2 y_2(y_2^2+y_1 y_2-8))\end{bmatrix}\right)\right\|^2 \\
		&\quad - 0.04 \cdot \left.\mathrm{trace}\left(\begin{bmatrix}a(1+t)&0\\0&b(2+t)\end{bmatrix}\right.^{-1}\cdot
		\left(\begin{bmatrix}a^2(3y_1^2+2y_1 y_2-8)&a^2 y_1^2\\b^2 y_2^2&b^2(3y_2^2+2y_1 y_2-8)\end{bmatrix}\right) \cdot\begin{bmatrix}a(1+t)&0\\0&b(2+t)\end{bmatrix}\right)\\
		&= \frac{6.25}{{(a(t + 1))}^2} \left( \frac{dy_1}{dt} + 0.04 a^2 y_1 (y_1^2 + y_1 y_2 - 8) \right)^2 +
		\frac{6.25}{{(b(t + 2))}^2}  \left( \frac{dy_2}{dt} + 0.04 b^2 y_2 (y_2^2 + y_1 y_2 - 8) \right)^2 \\
		&\quad - 0.04 (3a^2 y_1^2 + 3b^2 y_2^2 + 2(a^2 +b^2) y_1 y_2 - 8(a^2 +b^2)).
	\end{align*}
	Using the same method as Example $\ref{example 5.1}$, we can obtain the most probable path of system $\eqref{5.6}$ between two different metastable states (see $Fig.\ref{F.2}$). We define $P =\frac{a^2}{b^2}$. For different scales of $a$ and $b$, we can obtain the optimal paths with different $P$-values, which can be referred to $Fig.\ref{F.3}$ for details.

\begin{figure}[htbp]
	\centering
	\includegraphics[width=16.66cm]{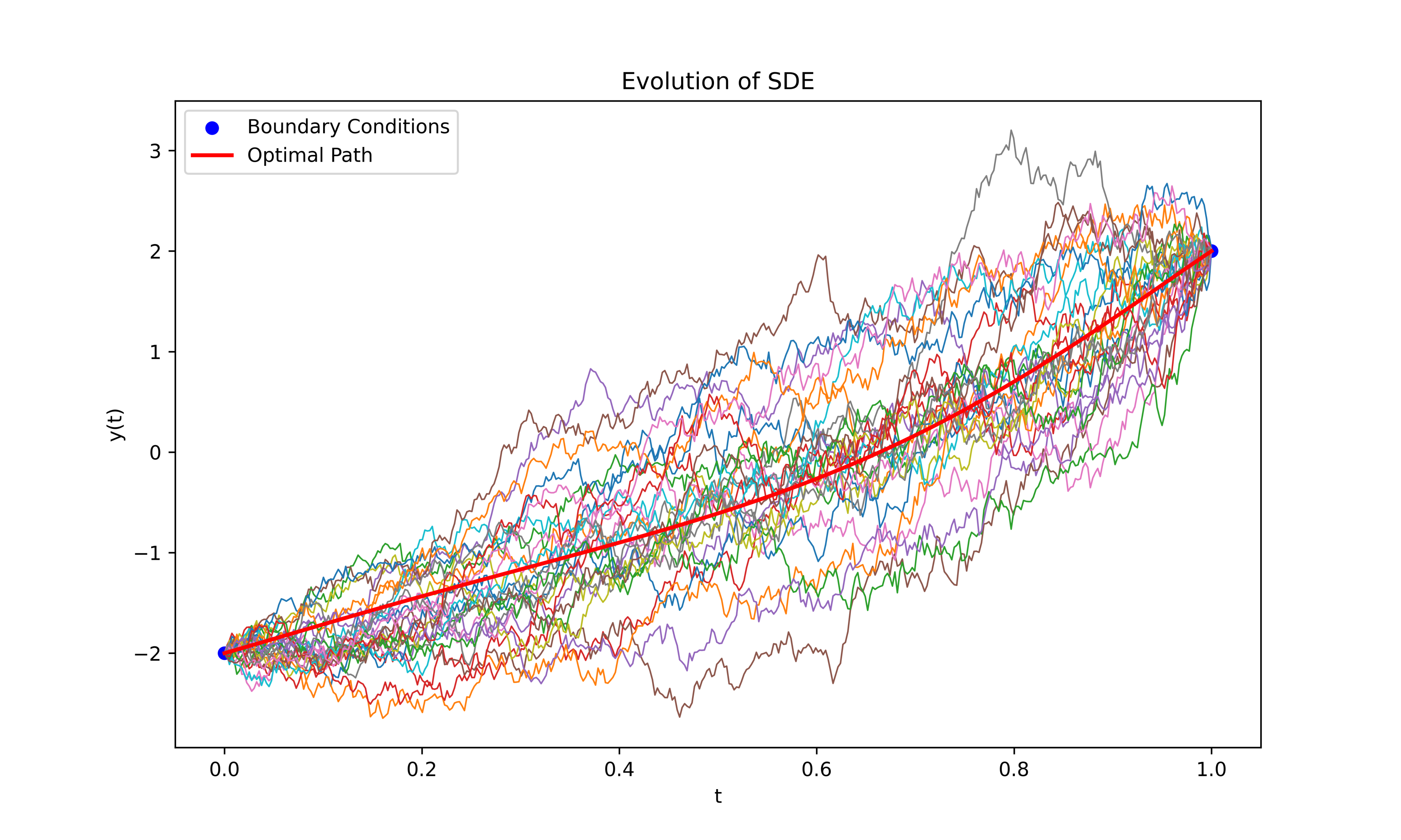}\\
	\includegraphics[width=16.66cm]{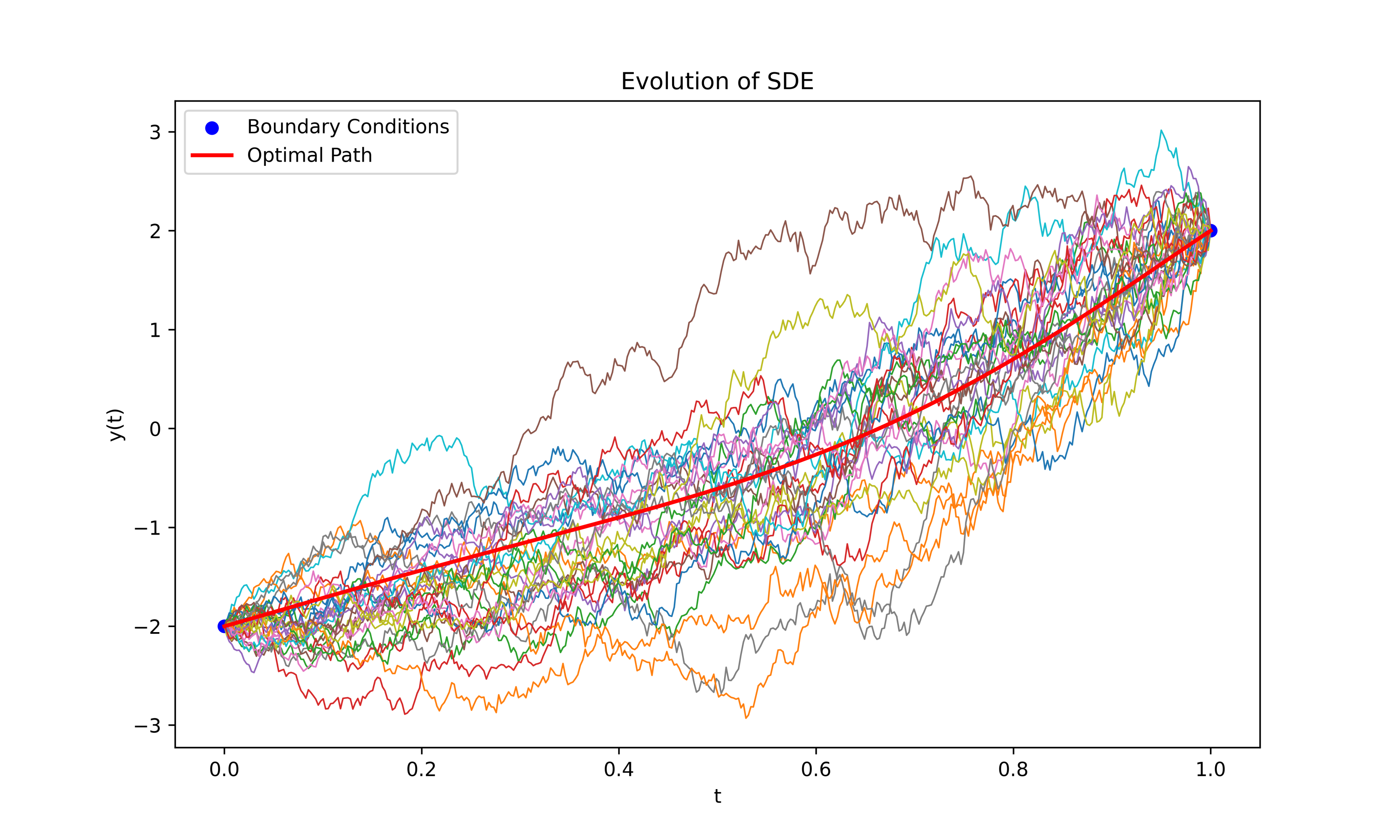}
	\caption{The stochastic process $x(t)$ (thin line) satisfies equation $\eqref{5.1}$ from metastable state $x=-2$ to metastable state $x=2$ and the most probable path $y(t)$ (thick red line) is generated by equation $\eqref{5.5}$. Due to its stochastics, we demonstrated the situation of two simulations.}
	\label{F.1}
\end{figure} 
\begin{figure}[htbp]
	\centering
	\includegraphics[width=12cm]{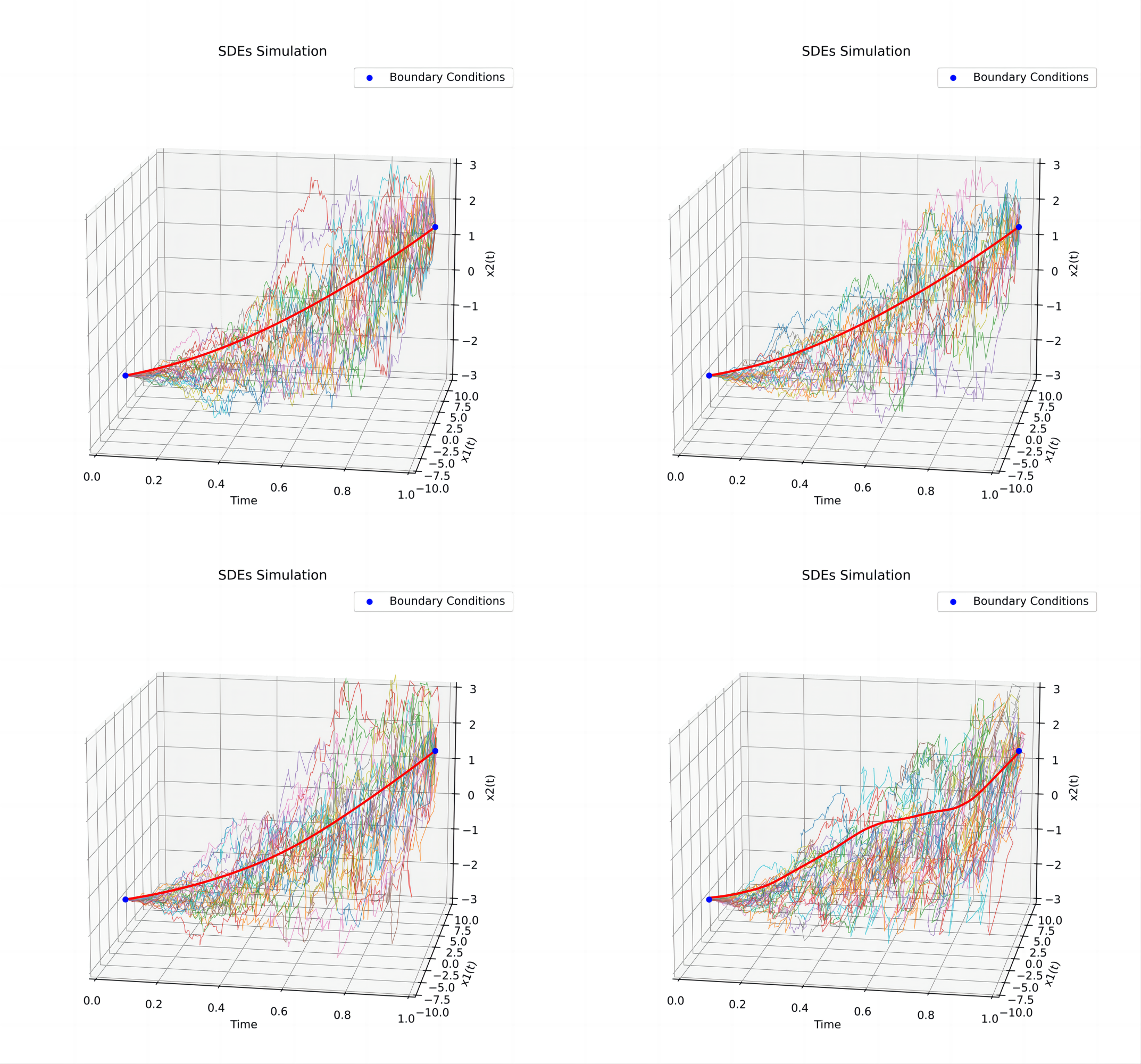}
	\caption{The stochastic process $x(t)$ (thin line) satisfies equation $\eqref{5.6}$ from metastable state $x=(-2,-2)$ to metastable state $x=(2,2)$ and the most probable path $y(t)$ (thick red line) is generated by minimizing the Onsager-Machlup functional.Here, parameters $b=1, a=1, 5, 10, 30$ correspond to the figures in the top left, top right, bottom left, and bottom right, respectively.}
	\label{F.2}
\end{figure} 
\begin{figure}[htbp]
	\centering
	\includegraphics[width=12cm]{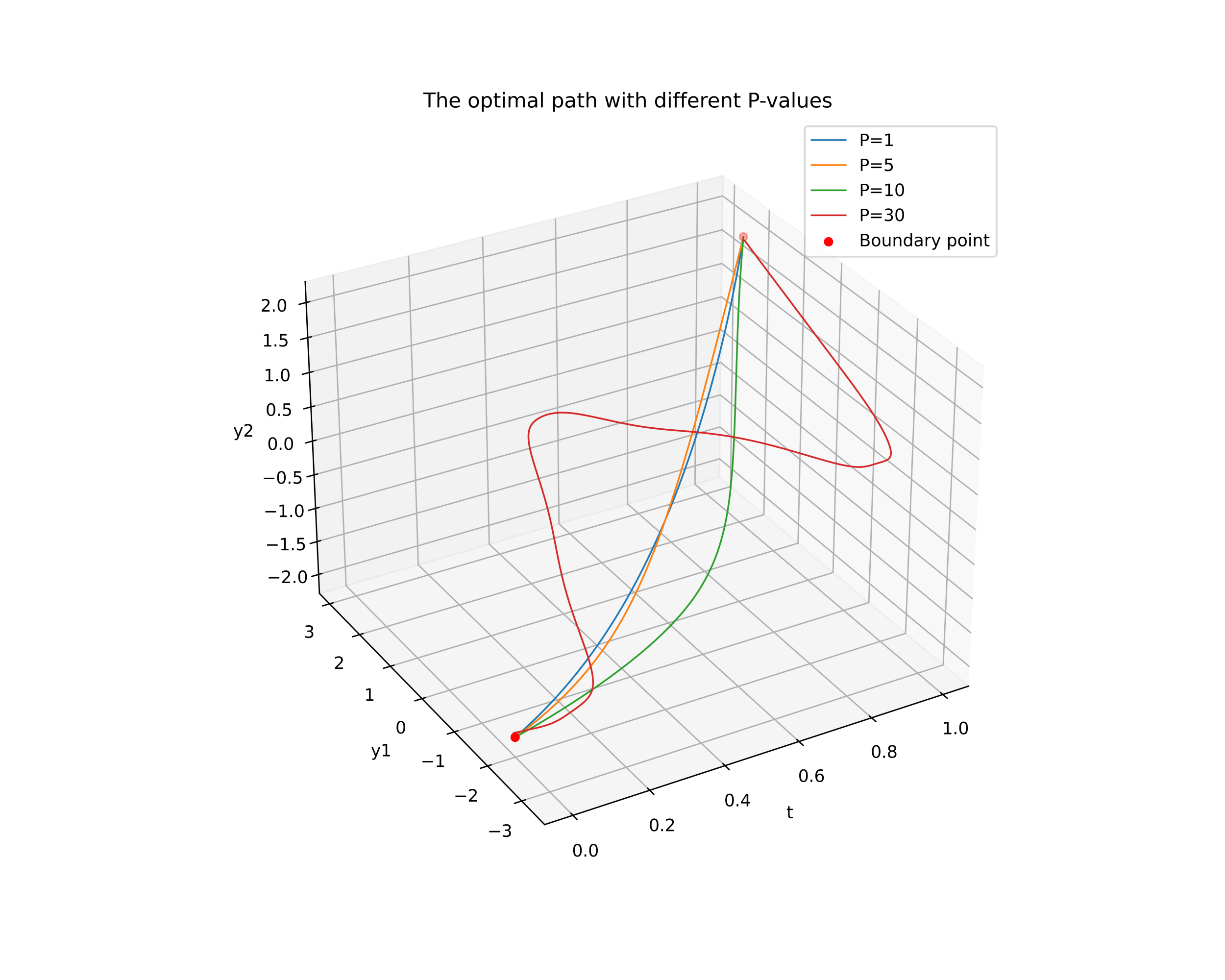}
	\caption{Comparison figure of different the most probable paths.}
	\label{F.3}
\end{figure}


\newpage


\newpage
\bibliographystyle{plain}
\bibliography{Ref}

\end{document}